%% file: arxiv-landom.tex
\documentclass[11pt]{article}

\usepackage[letterpaper,margin=1in]{geometry}
\usepackage{amsmath,amssymb,amsthm,bm}
\usepackage{graphicx}
\usepackage{xcolor}
\usepackage{algorithm}
\usepackage{algpseudocode}
\usepackage{placeins}
\usepackage{enumitem}
\usepackage[round,authoryear]{natbib}
\usepackage[hidelinks,bookmarks=true]{hyperref}
\bibpunct{(}{)}{;}{a}{,}{,}

\newtheorem{theorem}{Theorem}
\newtheorem{lemma}[theorem]{Lemma}
\newtheorem{proposition}[theorem]{Proposition}
\newtheorem{corollary}[theorem]{Corollary}
\newtheorem{assumption}[theorem]{Assumption}
\theoremstyle{remark}
\newtheorem{remark}[theorem]{Remark}

\newcommand{\algname}{Landom}
\newcommand{\algplus}{Landom+}

\input{macros.tex}


\title{{\bf Infeasible Momentum Orthogonalized Methods for Optimization on the Stiefel Manifold}}
\author{\begin{tabular}{@{}c@{\quad}c@{\quad}c@{\quad}c@{}}
Yuxuan Zeng$^{1}$ & Qi Deng$^{1}$ & Dongdong Ge$^{1}$ & Yinyu Ye$^{1,2,3}$\\
{\normalfont\scriptsize\texttt{zengyuxuan0522@sjtu.edu.cn}} &
{\normalfont\scriptsize\texttt{qdeng24@sjtu.edu.cn}} &
{\normalfont\scriptsize\texttt{ddge@sjtu.edu.cn}} &
{\normalfont\scriptsize\texttt{yinyu-ye@stanford.edu}}
\end{tabular}\\[4pt]
{\normalfont\small $^{1}$Shanghai Jiao Tong University}\\
{\normalfont\small $^{2}$Shanghai Institute for Mathematics and Interdisciplinary Sciences (SIMIS)}\\
{\normalfont\small $^{3}$Stanford University}}
\date{}

\begin{document}
\maketitle

\input{main.tex}
\input{appendix.tex}

\bibliography{sample}
\end{document}

%% file: macros.tex
\global\long\def\brbra#1{\big(#1\big)}%
\global\long\def\bsbra#1{\big[#1\big]}%
\global\long\def\Bsbra#1{\Big[#1\Big]}%
\global\long\def\bcbra#1{\big\{#1\big\}}%

\global\long\def\vertiii#1{\left\vert \kern-0.25ex  \left\vert \kern-0.25ex  \left\vert #1\right\vert \kern-0.25ex  \right\vert \kern-0.25ex  \right\vert }%

\global\long\def\Df{\mathop{\mathrm{D}\kern0pt}\nolimits}%

\global\long\def\skewop{\operatorname{skew}}%
\global\long\def\symop{\operatorname{sym}}%
\newcommand{\cpolar}{\operatorname{polar}}

\global\long\def\grad{\operatorname{grad}}%
\global\long\def\and{\mathrm{and}}%
\global\long\def\reals{\mathbb{R}}%

\global\long\def\Abf{\mathbf{A}}%
\global\long\def\Ubf{\mathbf{U}}%
\global\long\def\Pbf{\mathbf{P}}%
\global\long\def\Ibf{\mathbf{I}}%
\global\long\def\Ebf{\mathbf{E}}%
\global\long\def\Mbf{\mathbf{M}}%
\global\long\def\Pbf{\mathbf{P}}%
\global\long\def\Wbf{\mathbf{W}}%
\global\long\def\Bbf{\mathbf{B}}%
\global\long\def\Cbf{\mathbf{C}}%
\global\long\def\Dbf{\mathbf{D}}%
\global\long\def\Gbf{\mathbf{G}}%
\global\long\def\Hbf{\mathbf{H}}%
\global\long\def\Kbf{\mathbf{K}}%
\global\long\def\Sbf{\mathbf{S}}%
\global\long\def\Vbf{\mathbf{V}}%
\global\long\def\Xbf{\mathbf{X}}%
\global\long\def\Ybf{\mathbf{Y}}%
\global\long\def\Zbf{\mathbf{Z}}%

\global\long\def\Psibf{\boldsymbol{\Psi}}%
\global\long\def\Omegabf{\boldsymbol{\Omega}}%
\global\long\def\Sigmabf{\boldsymbol{\Sigma}}%
\global\long\def\Lambdabf{\boldsymbol{\Lambda}}%
\global\long\def\Deltabf{\boldsymbol{\Delta}}%

%% file: main.tex
\begin{abstract}
We study deterministic and stochastic optimization over the Stiefel manifold, explicitly allowing iterates to deviate from the manifold rather than enforcing feasibility via retraction at each step.
We introduce \algname{}, a Landing-type method that constructs momentum in the relative-gradient direction, maps this update to its compact polar factor, and combines the resulting tangential step with a step along the negative gradient of the orthogonality penalty to improve feasibility.
Under operator-to-nuclear smoothness, 
\algname{} achieves an \(\mathcal O(T^{-1/2})\) bound on 
the averaged nuclear norm of the relative gradient with 
exact gradients, and an \(\mathcal O(T^{-1/4})\) bound in 
expectation with stochastic momentum, bounded variance, and 
a fixed batch size. Our analysis also provides explicit 
pathwise bounds on the orthogonality residual, including 
its dependence on the step size and initial infeasibility. 
To address the coupling between objective and feasibility correction in \algname{}, we further propose \algplus{}, a two-timescale variant in which the tangential and normal step sizes are decoupled.
By keeping the normal step size fixed, \algplus{} achieves tighter control over the orthogonality residual following an initial transient while maintaining the respective convergence rates for stationarity.
Numerical experiments across a range of applications illustrate that the proposed algorithms remain competitive with established manifold optimization approaches.
\end{abstract}

\section{Introduction}
\label{sec:introduction}

We consider the following stochastic optimization problem over the Stiefel manifold \(\operatorname{St}(n,p)\) with \(n\ge p\):
\begin{equation}
\label{eq:stiefel-problem}
\min_{\Xbf\in\operatorname{St}(n,p)}\;
f(\Xbf)=\mathbb E_{\xi\sim\mathcal D}\bsbra{f(\Xbf;\xi)},
\qquad
\operatorname{St}(n,p)
:=
\bcbra{
\Xbf\in\mathbb R^{n\times p}:\Xbf^\top \Xbf=\Ibf_p
},
\end{equation}
where \(\xi\) is a random sample drawn from a distribution \(\mathcal D\), and \(f(\cdot\,;\xi):\mathcal U\to\mathbb R\) denotes the associated sample loss, defined on an open set \(\mathcal U\subseteq\mathbb R^{n\times p}\) containing \(\operatorname{St}(n,p)\).
Optimization over the Stiefel manifold arises naturally when matrix variables are required to satisfy orthonormality constraints. Representative applications include principal component analysis (PCA), independent component analysis (ICA), canonical correlation analysis, and the orthogonal Procrustes problem \citep{hotelling1933principal,hyvarinen1999fast,hotelling1936relations,schonemann1966generalized}. Beyond these classical statistical and matrix problems, orthonormality constraints have also been used in neural-network training to control weight geometry and improve training stability \citep{bansal2018orthogonality,wang2020orthogonal}. More recently, optimization over the Stiefel manifold has been used for parameter-efficient adaptation of large language models \citep{park-etal-2025-riemannian}, while broader manifold constraints have been studied in language-model pre-training \citep{an2026demystifying}. Across these settings, a common algorithmic challenge is to optimize the objective while effectively enforcing the nonlinear orthonormality constraint.

Standard Riemannian methods enforce the orthonormality constraint by constructing tangent directions and applying a retraction so that every iterate remains exactly feasible. Common choices include QR-, polar-, and Cayley-type retractions \citep{absil2008optimization,absil2012projection,wen2013feasible}. These operations introduce additional matrix factorizations, inverse square roots, or linear solves, whose cost can become substantial for large constrained matrices, particularly when gradient evaluation is not the dominant bottleneck \citep{ablin2024infeasible,xiong2026second}. A simpler alternative is squared orthogonality regularization, but finite penalty parameters do not in general enforce exact feasibility, whereas large penalties can deteriorate conditioning \citep{bansal2018orthogonality,ablin2024infeasible}. The Landing framework instead permits infeasible iterates and combines a tangential update for decreasing \(f\) with a step along the negative gradient of the orthogonality penalty to drive the iterates toward the manifold; it was introduced for the orthogonal group by \citet{pmlr-v151-ablin22a} and later extended to the Stiefel manifold by \citet{ablin2024infeasible}. This viewpoint separates objective decrease from feasibility restoration without requiring exact orthogonalization at every step. Other infeasible approaches for orthogonality constraints include the parallelizable PLAM and PCAL methods of \citet{gao2019parallelizable}, which update the Lagrange multipliers by an explicit formula instead of a dual iteration; smooth exact penalty and constraint dissolving formulations \citep{xiao2022class,xiao2024dissolving}, which replace the constrained problem by an unconstrained one whose stationary points coincide with those of the original problem near the manifold; and the orthogonal directions constrained gradient method (ODCGM) of \citet{schechtman2023orthogonal}, which requires only a projection onto a linear space at each step and recovers the Landing algorithm on the Stiefel manifold. Like Landing, these methods construct their descent directions in the Frobenius geometry.

Once feasibility restoration is assigned a dedicated geometric role, a complementary question is how to choose the tangential update. Normalized first-order methods separate update direction from gradient magnitude \citep{hazan2015beyond,levy2016power,pmlr-v119-cutkosky20b}. More generally, direction selection can be formulated through norm-induced steepest descent or a linear minimization oracle (LMO) over a norm ball, where the chosen norm determines the update geometry \citep{bernstein2024old,pmlr-v267-pethick25a}. For matrix variables, matrix norms can retain structural information beyond the default Frobenius geometry. In particular, the spectral norm is dual to the nuclear norm, and the maximally aligned linear-minimization direction is the polar factor. Muon is a modern optimizer built around this spectral geometry \citep{jordan2024muon}, and recent analysis indicates that its advantage over gradient descent can depend on structural properties such as low-rank Hessians rather than holding universally \citep{shen2026muon}. This motivates studying how polar-factor updates can be integrated with infeasible Stiefel optimization, where feasibility must be controlled dynamically rather than enforced by a retraction.

Recent work has investigated spectral-norm-based first-order methods for optimization under orthonormality constraints. Practical Manifold Muon formulations solve a spectral-norm tangent subproblem through iterative inner methods, using dual ascent or the alternating direction method of multipliers (ADMM), followed by retraction \citep{bernstein2025manifolds,buchanan2025mmuonadmm}. Formal developments include SPEL \citep{yang2026manifold}, iMuon \citep{li2026intrinsic}, and Skewon \citep{solonko2026skewon}, which construct spectral tangent directions while maintaining or restoring feasibility. In parallel, \citet{zhang2026retraction} study retraction-free Stiefel optimization with a decaying projected-gradient step and a fixed step size for the normalizing component, but without a tangential update based on the polar factor of relative-gradient momentum.
To our knowledge, these spectral methods either keep the iterates on the manifold through a retraction or projection, or do not analyze how a polar-factor tangential step interacts with feasibility restoration. In particular, it remains unclear whether a polar-factor update of stochastic relative-gradient momentum can be combined with Landing-type feasibility restoration while still admitting explicit bounds on the orthogonality residual.

To address this gap, we study retraction-free optimization under orthonormality constraints using polar-factor updates, focusing on the interaction between objective decrease and feasibility restoration in deterministic and stochastic settings. Our contributions are threefold.

First, we propose \algname{}, a Landing algorithm with orthogonalized momentum that searches a neighborhood of the Stiefel manifold without requiring a retraction at every step. Inspired by Muon, \algname{} forms relative-gradient momentum, maps it to its compact polar factor, and combines the resulting tangential update with a step along the negative gradient of the orthogonality penalty. This differs from applying a spectral LMO to tangent or Riemannian directions, as in MACRO and MCSD \citep{an2026demystifying,yang2026manifold}. Because relative-gradient momentum and its polar factor are skew-symmetric, the directional derivative of the Gram map \(\Xbf\mapsto \Xbf^\top \Xbf\) along the tangential update vanishes, even away from the manifold. We establish deterministic and stochastic convergence guarantees for \algname{}, together with explicit bounds on the orthogonality residual.

Second, we introduce \algplus{}, a two-timescale extension motivated by a limitation of \algname{}: reducing its coupled step size simultaneously shrinks the tangential update used to decrease \(f\) and the step along the negative penalty gradient used to restore feasibility. Building on the use of separate tangential and normal steps in related work \citep{lu2026variance,zhang2026retraction}, \algplus{} retains a fixed normal step while allowing the tangential step to shrink. Our analysis shows that this separation, combined with the skew-symmetric polar-factor update, yields tighter orthogonality-residual bounds and a shorter feasibility-restoration phase from infeasible initialization than \algname{}, while preserving the leading stationarity rates. The corresponding convergence rates are summarized in Table~\ref{tab:coupled-separate-steps-comparison}.

Third, we evaluate the proposed methods on linear optimization over the Stiefel manifold, stochastic PCA, real-data ICA, and orthogonally constrained neural networks. The experiments support the predicted residual scalings and demonstrate the optimization benefits of polar-factor updates. In stochastic PCA, \algname{} approaches stationarity faster than the tested Landing-gradient and retraction-based baselines under matched computation budgets. Across the experiments, \algplus{} provides tighter feasibility control while retaining competitive optimization performance and computational efficiency.

\begin{table}[!hbp]
\centering
\setlength{\tabcolsep}{4pt}
\begin{tabular}{@{}llcc@{}}
\hline
Regime & Variant & Tail stationarity & Tail residual \\
\hline
Stochastic & \algname{} & $\mathcal O(T^{-1/4})$ & $\mathcal O(T^{-3/4})$ \\
Stochastic & \algplus{} & $\mathcal O(T^{-1/4})$ & $\mathcal O(T^{-3/2})$ \\
Deterministic & \algname{} & $\mathcal O(T^{-1/2})$ & $\mathcal O(T^{-1/2})$ \\
Deterministic & \algplus{} & $\mathcal O(T^{-1/2})$ & $\mathcal O(T^{-1})$ \\
\hline
\end{tabular}
\caption{Tail stationarity and feasibility rates for \algname{} and \algplus{}.}
\label{tab:coupled-separate-steps-comparison}
\end{table}

\section{Preliminaries}
\label{sec:preliminaries}

\paragraph{Notation.}
For \(\Abf,\Bbf\in\reals^{m\times q}\), \(\langle \Abf,\Bbf\rangle:=\operatorname{tr}(\Abf^\top \Bbf)\)
denotes the Frobenius inner product.
We denote by \(\|\cdot\|_F\), \(\|\cdot\|_{\mathrm{op}}\), and \(\|\cdot\|_*\)
the Frobenius norm, the operator norm, and the nuclear norm, respectively.
For a square matrix \(\Bbf\), \(\skewop(\Bbf):=\frac12(\Bbf-\Bbf^\top)\) and
\(\symop(\Bbf):=\frac12(\Bbf+\Bbf^\top)\) are its skew-symmetric and symmetric parts, which
are orthogonal under \(\langle\cdot,\cdot\rangle\).
We write \(\mathbb S^p\) for the symmetric \(p\times p\) matrices,
\(\operatorname{Skew}(n)\) for the skew-symmetric \(n\times n\) matrices,
and \(\Ibf_p\) for the identity matrix.
Bold symbols denote matrices. Matrix-valued maps are denoted by non-bold
symbols, such as \(C(\Xbf)\) and \(\Psi(\Xbf)\) below, and their values at an
iterate \(\Xbf_t\) are denoted by \(\Cbf_t\) and \(\Psibf_t\).
\([u]_+:=\max\{u,0\}\) denotes the positive part of a scalar \(u\).

To quantify infeasibility, we define two associated scalar measures
\(d(\Xbf):=\|C(\Xbf)\|_{\mathrm{op}}\) and \(N(\Xbf):=\frac14\|C(\Xbf)\|_F^2\), where \(C(\Xbf):=\Xbf^\top\Xbf-\Ibf_p\).
It is easy to verify that \(C(\Xbf)\in\mathbb S^p\), and
\(\frac14d(\Xbf)^2\le N(\Xbf)\le\frac p4d(\Xbf)^2\).
The gradient of the orthogonality penalty \(N\) is \(\nabla N(\Xbf)=\Xbf C(\Xbf)\). A step along \(-\nabla N(\Xbf)\) reduces \(N\) for a sufficiently small step size whenever \(\nabla N(\Xbf)\neq0\); we refer to \(\nabla N(\Xbf)\) as the normalizing component of the update direction.

Let \(\varepsilon\in(0,1)\) and define the safe region \(\mathcal S_\varepsilon:=\bcbra{\Xbf\in\mathbb R^{n\times p}:d(\Xbf)\le\varepsilon}\).
We further assume that $f(\Xbf)$ is well-defined on $\mathcal S_\varepsilon$, namely, 
$\mathcal S_\varepsilon\subseteq\mathcal U$.
The following lemma records the corresponding
spectral bounds.

\begin{lemma}[Bounds in the safe region]
\label{lem:safe-region-bounds}
For every \(\Xbf\in\mathcal S_\varepsilon\), the following bounds hold, where \(\sigma_i(\Xbf)\) denotes the \(i\)-th singular value of \(\Xbf\):
\((1-\varepsilon)\Ibf_p\preceq\Xbf^\top\Xbf\preceq(1+\varepsilon)\Ibf_p\), \(\sqrt{1-\varepsilon}\le\sigma_i(\Xbf)\le\sqrt{1+\varepsilon}\) for \(i=1,\ldots,p\), \(\|\Xbf\|_{\mathrm{op}}\le\sqrt{1+\varepsilon}\), and \(\|\Xbf\|_F\le\sqrt{p(1+\varepsilon)}\).
In particular, \(\Xbf\) has full column rank and
\(\mathcal S_\varepsilon\) is compact.
\end{lemma}

The proof follows from the spectral characterization of the operator
norm and is deferred to Appendix~\ref{app:shared-safe-region}.
For any \(\Xbf\in\mathcal S_\varepsilon\), let \(\mathcal M_{\Xbf}:=\bcbra{\Ybf\in\mathbb R^{n\times p}:\Ybf^\top\Ybf=\Xbf^\top\Xbf}\).
By Lemma~\ref{lem:safe-region-bounds},
\(\Xbf^\top \Xbf\succeq(1-\varepsilon)\Ibf_p\succ0\).
Propositions~1 and~2 of \citet{gao2022optimization} imply that
\(\mathcal M_{\Xbf}\) is a closed embedded submanifold and that, for every
\(\Ybf\in\mathcal M_{\Xbf}\),
\(T_{\Ybf}\mathcal M_{\Xbf}=\{\Zbf\in\reals^{n\times p}:\Ybf^\top\Zbf+\Zbf^\top\Ybf=0\}=\{\Omegabf\Ybf:\Omegabf\in\operatorname{Skew}(n)\}\)
is its tangent space at \(\Ybf\). The normal space at \(\Ybf\) under the Frobenius inner product is \(N_{\Ybf}\mathcal M_{\Xbf}=\{\Ybf \Sbf:\Sbf\in\mathbb S^p\}\).
At \(\Ybf=\Xbf\), these identities give \(\Omegabf\Xbf\in T_{\Xbf}\mathcal M_{\Xbf}\) and \(\Xbf C(\Xbf)\in N_{\Xbf}\mathcal M_{\Xbf}\), hence \(\langle\Omegabf\Xbf,\Xbf C(\Xbf)\rangle=0\) for every \(\Omegabf\in\operatorname{Skew}(n)\).

Following \citet{pmlr-v151-ablin22a}, we define the \emph{relative gradient} for each \(\Xbf \in \mathcal U\) as follows:
\(\Psi(\Xbf):=\skewop(\nabla f(\Xbf)\Xbf^\top)\in\operatorname{Skew}(n)\).
For \(\Xbf\in\mathcal S_\varepsilon\), under the canonical metric used by \citet[Proposition~4]{gao2022optimization},
the Riemannian gradient of \(f\) on \(\mathcal M_{\Xbf}\) at \(\Xbf\) is
\(2\Psi(\Xbf)\Xbf\). Following the convention of
\citet{ablin2024infeasible}, we rescale this metric by a factor of two and
define the Riemannian gradient at \(\Xbf\) by
\(\grad f(\Xbf):=\Psi(\Xbf)\Xbf\).

We use \(\|\Psi(\Xbf)\|_*\) as the primary stationarity measure and 
 \(d(\Xbf)\) as the feasibility measure. 
It is easy to verify that, for any \(\Xbf\in\operatorname{St}(n,p)\),
\(\|\grad f(\Xbf)\|_F=\|\Psi(\Xbf)\Xbf\|_F\le\|\Psi(\Xbf)\|_F\le\|\Psi(\Xbf)\|_*\).
Specifically, nuclear-norm stationarity ($\Psi(\Xbf)=0$) implies the usual Riemannian stationarity ($\grad f(\Xbf)=0$) at feasible points. 
The derivations of the corresponding identities can be found in Appendix~\ref{app:shared-level-set-identities}.

For \(\Hbf\in\mathbb R^{n\times n}\) with
\(r:=\operatorname{rank}(\Hbf)>0\), let
\(\Hbf=\Ubf_r\Sigmabf_r\Vbf_r^\top\) be the compact singular value
decomposition (SVD), where \(\Ubf_r,\Vbf_r\in\mathbb R^{n\times r}\)
have orthonormal columns and
\(\Sigmabf_r\in\mathbb R^{r\times r}\) contains the positive singular
values on its diagonal. We define the compact polar factor by
\(\cpolar(\Hbf):=\Ubf_r\Vbf_r^\top\), with \(\cpolar(0):=0\).
The following lemma summarizes its norm, alignment, and
skew-symmetry properties.
\begin{lemma}[Compact polar factor properties]
\label{lem:matrix-sign-properties}
Fix \(\Hbf\in\mathbb R^{n\times n}\).
Then
\[
\|\cpolar(\Hbf)\|_{\mathrm{op}}\le1,
\qquad
\langle \Hbf,\cpolar(\Hbf)\rangle=\|\Hbf\|_*,
\qquad
\cpolar(\Hbf)\in\underset{\|\Abf\|_{\mathrm{op}}\le1}{\arg\max}\langle \Hbf,\Abf\rangle.
\]
If \(\Hbf\in\operatorname{Skew}(n)\), then
\(\cpolar(\Hbf)\in\operatorname{Skew}(n)\), and
\[
\cpolar(\Hbf)\in
\underset{\Abf\in\operatorname{Skew}(n),\,\|\Abf\|_{\mathrm{op}}\le1}{\arg\max}
\langle \Hbf,\Abf\rangle.
\]
\end{lemma}

The following assumptions are used in the convergence analyses of our proposed methods. The smoothness condition relates the nuclear norm of the gradient difference to the operator norm of the displacement, and the stochastic oracle condition is needed only for the stochastic results.

\begin{assumption}[Smoothness]
\label{ass:operator-nuclear-smoothness}
There exist an open convex set
\(\mathcal U_\varepsilon\) satisfying
\(\mathcal S_\varepsilon\subseteq\mathcal U_\varepsilon\subseteq
\mathcal U\) and a constant \(L_*>0\) such that
\[
\|\nabla f(\Xbf)-\nabla f(\Ybf)\|_*
\le
L_*\|\Xbf-\Ybf\|_{\mathrm{op}},
\qquad
\Xbf,\Ybf\in\mathcal U_\varepsilon.
\]
\end{assumption}

\begin{assumption}[Stochastic first-order oracle]
\label{ass:stochastic-oracle}
There exists \(\sigma\ge0\) such that, for every
\(\Xbf\in\mathcal S_\varepsilon\) and \(\xi\sim\mathcal D\),
\[
\mathbb E_{\xi}[\nabla f(\Xbf;\xi)]
=
\nabla f(\Xbf),
\qquad
\mathbb E_{\xi}
\bsbra{
\|\nabla f(\Xbf;\xi)-\nabla f(\Xbf)\|_F^2
}
\le
\sigma^2.
\]
\end{assumption}

The following proposition collects the estimates used to analyze objective descent, momentum tracking, and feasibility for both algorithms.

\begin{proposition}
\label{prop:shared-basic-estimates}
Suppose Assumptions~\ref{ass:operator-nuclear-smoothness}
and~\ref{ass:stochastic-oracle} hold.
Then the following statements hold.
\begin{enumerate}[label=\textnormal{(\roman*)},ref=\textnormal{(\roman*)},leftmargin=*,widest=iii]
\item\label{prop:shared-basic-estimates-descent}
For all \(\Xbf,\Ybf\in\mathcal S_\varepsilon\),
\[
f(\Ybf)
\le
f(\Xbf)
+
\langle\nabla f(\Xbf),\Ybf-\Xbf\rangle
+
\frac{L_*}{2}\|\Ybf-\Xbf\|_{\mathrm{op}}^2.
\]

\item\label{prop:shared-basic-estimates-psi-lipschitz}
Let \(G_\varepsilon:=\max_{\Xbf\in\mathcal S_\varepsilon}\|\nabla f(\Xbf)\|_*\) and \(L_{\Psi,\varepsilon}:=L_*\sqrt{1+\varepsilon}+G_\varepsilon\). Then \(\|\Psi(\Xbf)-\Psi(\Ybf)\|_*\le L_{\Psi,\varepsilon}\|\Xbf-\Ybf\|_{\mathrm{op}}\) for all \(\Xbf,\Ybf\in\mathcal S_\varepsilon\).

\item\label{prop:shared-basic-estimates-sampling}
For any fixed \(\Xbf\in\mathcal S_\varepsilon\), draw \(\xi_1,\ldots,\xi_B\overset{\mathrm{iid}}{\sim}\mathcal D\) and define \(\Gbf^{(B)}:=B^{-1}\sum_{i=1}^B\nabla f(\Xbf;\xi_i)\) and \(\Psibf^{(B)}:=\skewop(\Gbf^{(B)}\Xbf^\top)\). With expectations taken over these samples, \(\mathbb E[\Gbf^{(B)}]=\nabla f(\Xbf)\), \(\mathbb E[\Psibf^{(B)}]=\Psi(\Xbf)\), and
\[
\mathbb E\|\Gbf^{(B)}-\nabla f(\Xbf)\|_F^2\le\frac{\sigma^2}{B},
\qquad
\mathbb E\|\Psibf^{(B)}-\Psi(\Xbf)\|_F^2\le(1+\varepsilon)\frac{\sigma^2}{B}.
\]
Consequently, \(\mathbb E\|\Psibf^{(B)}-\Psi(\Xbf)\|_*\le\sigma\sqrt{n(1+\varepsilon)/B}\).

\item\label{prop:shared-basic-estimates-coupling}
Let \(G^{\mathrm{sym}}_\varepsilon:=\max_{\Xbf\in\mathcal S_\varepsilon}\|\symop(\Xbf^\top\nabla f(\Xbf))\|_*\). Then \(|\langle\nabla f(\Xbf),\Xbf C(\Xbf)\rangle|\le G^{\mathrm{sym}}_\varepsilon d(\Xbf)\) for all \(\Xbf\in\mathcal S_\varepsilon\).
\end{enumerate}
The constants \(G_\varepsilon\) and \(G^{\mathrm{sym}}_\varepsilon\) are finite by continuity of \(\nabla f\) and compactness of \(\mathcal S_\varepsilon\).
\end{proposition}

The proof is given in Appendix~\ref{app:shared-basic-estimates-proof}.

\section{\algname{}: Landing with Orthogonalized Momentum}
\label{sec:algorithmic-framework}

Algorithm~\ref{alg:landom} presents \algname{}, which combines polar-factor momentum updates with a normal correction weighted by \(\lambda>0\), using a common step size.

At each iteration, Algorithm~\ref{alg:landom} draws a mini-batch of iid samples $\xi_{t,1},\ldots,\xi_{t,B}\sim\mathcal D$, independently of the initialization and all previous samples, computes the corresponding mini-batch gradient \(\Gbf_t^{(B)}\), and constructs the stochastic relative gradient \(\Psibf_t^{(B)}\).
This stochastic relative gradient is then incorporated into the exponential moving average (EMA) \(\Hbf_t\) using the momentum parameter \(\beta \in [0,1)\).
We denote by \(\Psibf_t := \Psi(\Xbf_t)\) the full relative gradient at iteration \(t\).
Because the stochastic relative gradients are skew-symmetric and the space \(\operatorname{Skew}(n)\) is closed under linear combinations, the exponential moving average \(\Hbf_t\) also remains skew-symmetric throughout the iterations.
By applying the compact polar map, we obtain \(\Omegabf_t \in \operatorname{Skew}(n)\), which satisfies \(\|\Omegabf_t\|_{\mathrm{op}} \leq 1\) in accordance with Lemma~\ref{lem:matrix-sign-properties}.
Consequently, the product \(\Omegabf_t \Xbf_t\) lies in the tangent space \(T_{\Xbf_t}\mathcal M_{\Xbf_t}\).

To complete the update, we add the gradient of the orthogonality penalty, weighted by \(\lambda\), to this tangential direction. Writing \(\Cbf_t:=C(\Xbf_t)\), we obtain
\begin{equation}
\label{eq:coupled-direction}
\Dbf_t:=\Omegabf_t\Xbf_t+\lambda \Xbf_t\Cbf_t,
\qquad
\Xbf_{t+1}=\Xbf_t-\tau_t\Dbf_t.
\end{equation}
The normal component \(\Xbf_t\Cbf_t=\nabla N(\Xbf_t)\) lies in \(N_{\Xbf_t}\mathcal M_{\Xbf_t}\), and the corresponding step sizes in the tangential and normal directions are given by \(\tau_t\) and \(\lambda\tau_t\), respectively.
We choose \(\tau_t\) as the minimum of two quantities: the target step size \(\eta_t^{\mathrm{tar}}>0\), which is chosen to promote convergence, and the safeguarding step size \(\eta_t^{\mathrm{safe}}=\eta_\varepsilon^{\mathrm{safe}}(d_t,m_t;\lambda)\), a function of the orthogonality residual \(d_t=\|\Cbf_t\|_{\mathrm{op}}\) and the direction magnitude \(m_t=\|\Dbf_t\|_{\mathrm{op}}\), which keeps the iterates in \(\mathcal S_\varepsilon\).
Details regarding the specific schedules for these step sizes are provided in Section~\ref{sec:coupled-convergence}. In the deterministic setting, \algname{} uses exact gradients with any fixed \(\beta\in[0,1)\); choosing \(\beta=0\) yields the no-momentum variant.

\begin{algorithm}[t]
\caption{\algname{}}
\label{alg:landom}
\begin{algorithmic}[1]
\Require $\Xbf_0\in\mathcal S_\varepsilon$, horizon $T$, batch size $B$,
momentum $\beta\in[0,1)$, penalty weight $\lambda>0$,
and target step sizes $\{\eta_t^{\mathrm{tar}}\}$
\For{$t=0,\ldots,T-1$}
    \State Draw $\xi_{t,1},\ldots,\xi_{t,B}\overset{\mathrm{iid}}{\sim}\mathcal D$  
    \State compute $\Gbf_t^{(B)}\gets B^{-1}\sum_{i=1}^B\nabla f(\Xbf_t;\xi_{t,i})$
    \State $\Psibf_t^{(B)}\gets\skewop(\Gbf_t^{(B)}\Xbf_t^\top)$
    \State $\Hbf_t\gets\Psibf_t^{(B)}$ if $t=0$; otherwise
    $\Hbf_t\gets\beta \Hbf_{t-1}+(1-\beta)\Psibf_t^{(B)}$
    \State $\Omegabf_t\gets\cpolar(\Hbf_t)$
    \label{line:coupled-msign}
    \State $\Cbf_t\gets \Xbf_t^\top \Xbf_t-\Ibf_p$, $d_t\gets\|\Cbf_t\|_{\mathrm{op}}$
    \label{line:coupled-residual}
    \State $\Dbf_t\gets\Omegabf_t\Xbf_t+\lambda \Xbf_t\Cbf_t$,
    $m_t\gets\|\Dbf_t\|_{\mathrm{op}}$
    \State $\eta_t^{\mathrm{safe}}\gets
    \eta_\varepsilon^{\mathrm{safe}}(d_t,m_t;\lambda)$
    \label{line:coupled-local-safeguard}
    \State $\tau_t\gets\min\{\eta_t^{\mathrm{tar}},\eta_t^{\mathrm{safe}}\}$
    \label{line:coupled-actual-step}
    \State $\Xbf_{t+1}\gets \Xbf_t-\tau_t\Dbf_t$
    \label{line:coupled-update}
\EndFor
\State \Return $\Xbf_T$
\end{algorithmic}
\end{algorithm}

\begin{remark}
\label{rem:skew-msign-order}
Since \(\Hbf_t\in\operatorname{Skew}(n)\), Lemma~\ref{lem:matrix-sign-properties} ensures that \(\Omegabf_t=\cpolar(\Hbf_t)\) is skew-symmetric, satisfies \(\|\Omegabf_t\|_{\mathrm{op}}\le1\), and achieves the exact alignment \(\langle\Hbf_t,\Omegabf_t\rangle=\|\Hbf_t\|_*\). Consequently,
\begin{equation}
\label{eq:relative-gradient-polar-alignment}
\langle\Psibf_t,\Omegabf_t\rangle
\ge \|\Psibf_t\|_*-2\|\Hbf_t-\Psibf_t\|_*.
\end{equation}
Hence, the compact polar factor preserves the nuclear-norm alignment up to the tracking error \(\|\Hbf_t-\Psibf_t\|_*\); exact alignment is recovered when \(\Hbf_t=\Psibf_t\). The convergence analysis therefore reduces to controlling this tracking error. This also motivates applying the skew projection before the compact polar factor, since reversing the order does not, in general, retain the exact nuclear-norm alignment with the skew component. The alignment estimate is proved in Appendix~\ref{app:momentum-alignment-proof}.
\end{remark}

\section{Convergence of \algname{}}
\label{sec:coupled-convergence}

Under the assumptions in Section~\ref{sec:preliminaries}, we first establish safe-region invariance and a pathwise recursion for the orthogonality residual, and then derive finite-time stationarity and feasibility guarantees.

\subsection{Safe-Region Invariance and Feasibility Dynamics}
\label{subsec:coupled-feasibility}

Expanding the update in Algorithm~\ref{alg:landom} and using \(\Omegabf_t\in\operatorname{Skew}(n)\) gives, pathwise,
\begin{equation}
\label{eq:coupled-gram-identity}
\Cbf_{t+1}=\Cbf_t-2\lambda\tau_t(\Cbf_t+\Cbf_t^2)+\tau_t^2\Dbf_t^\top\Dbf_t.
\end{equation}
By skew-symmetry, the terms linear in \(\tau_t\) arising from the tangential component cancel, and the remaining linear term comes entirely from the orthogonality penalty gradient. In contrast, for the residual \(\Xbf\Xbf^\top-\Ibf_n\) used in the discrete analysis of \citet{pmlr-v151-ablin22a}, a tangential direction \(\Abf\Xbf\) with \(\Abf\in\operatorname{Skew}(n)\) contributes the first-order commutator term \([\Abf,\Xbf\Xbf^\top-\Ibf_n]\), which does not vanish in general. The cancellation in \eqref{eq:coupled-gram-identity} therefore comes from the skew-symmetric left action combined with \(C(\Xbf)=\Xbf^\top\Xbf-\Ibf_p\), not from the compact polar map.

Recall that \(d_t:=\|\Cbf_t\|_{\mathrm{op}}\) and \(m_t:=\|\Dbf_t\|_{\mathrm{op}}\). For \(d\in[0,\varepsilon]\) and \(m\ge0\), define
\begin{equation}
\label{eq:coupled-safeguard}
\eta_{\varepsilon}^{\mathrm{safe}}(d,m;\lambda):=
\begin{cases}
\displaystyle\min\left\{
\frac{\lambda d(1-d)+\sqrt{\lambda^2d^2(1-d)^2+m^2(\varepsilon-d)}}{m^2},
\frac1{2\lambda}\right\}, & m>0,\\[1ex]
\displaystyle\frac1{2\lambda}, & m=0.
\end{cases}
\end{equation}
A uniform lower bound is obtained from the constants
\begin{equation}
\label{eq:coupled-uniform-safe-threshold}
M_\varepsilon:=\sqrt{1+\varepsilon}(1+\lambda\varepsilon),\qquad
\eta_\varepsilon^\star:=\min\left\{\frac1{2\lambda},\frac{\sqrt{\varepsilon/2}}{M_\varepsilon},\frac{\lambda\varepsilon(1-\varepsilon)}{M_\varepsilon^2}\right\}.
\end{equation}

\begin{lemma}[Safeguard and safe-region invariance]
\label{lem:coupled-safeguard}
Let \(\Xbf\in\mathcal S_\varepsilon\) and \(\Omegabf\in\operatorname{Skew}(n)\), and define \(\Cbf:=\Xbf^\top\Xbf-\Ibf_p\), \(\Dbf:=\Omegabf\Xbf+\lambda\Xbf\Cbf\), \(d:=\|\Cbf\|_{\mathrm{op}}\), and \(m:=\|\Dbf\|_{\mathrm{op}}\). Then every \(0\le\tau\le\eta_\varepsilon^{\mathrm{safe}}(d,m;\lambda)\) satisfies
\[
\|(\Xbf-\tau\Dbf)^\top(\Xbf-\tau\Dbf)-\Ibf_p\|_{\mathrm{op}}\le\varepsilon.
\]
For the directions generated by Algorithm~\ref{alg:landom}, \(\Omegabf_t=\cpolar(\Hbf_t)\) is skew-symmetric and satisfies \(\|\Omegabf_t\|_{\mathrm{op}}\le1\). Hence, whenever \(\Xbf_t\in\mathcal S_\varepsilon\), we have \(m_t\le M_\varepsilon\), \(\eta_t^{\mathrm{safe}}\ge\eta_\varepsilon^\star\), and \(\Xbf_{t+1}\in\mathcal S_\varepsilon\) pathwise. Consequently, \(\Xbf_0\in\mathcal S_\varepsilon\) implies \(\Xbf_t\in\mathcal S_\varepsilon\) for all \(t\ge0\).
\end{lemma}

For the finite-time analysis, consider deterministic target step sizes satisfying
\begin{equation}
\label{eq:coupled-target-step-condition}
0<\eta_t^{\mathrm{tar}}\le\eta_\varepsilon^\star,\qquad t=0,\ldots,T-1.
\end{equation}
The safeguard is then inactive and \(\tau_t=\eta_t^{\mathrm{tar}}\) pathwise. Set \(a_\varepsilon:=2\lambda(1-\varepsilon)\), \(b_\varepsilon:=M_\varepsilon^2\), and \(\Gamma_\varepsilon:=b_\varepsilon/a_\varepsilon\).

\begin{lemma}[Pathwise orthogonality-residual recursion]
\label{lem:coupled-residual-recursion}
Suppose that \(\Xbf_0\in\mathcal S_\varepsilon\) and that the deterministic target step sizes satisfy \eqref{eq:coupled-target-step-condition}. Then, for \(t=0,\ldots,T-1\),
\begin{equation}
\label{eq:coupled-residual-recursion}
d_{t+1}\le(1-a_\varepsilon\tau_t)d_t+b_\varepsilon\tau_t^2.
\end{equation}
Moreover, \(1-a_\varepsilon\tau_t\ge\varepsilon>0\), and consequently
\begin{equation}
\label{eq:coupled-feasibility-telescoping}
a_\varepsilon\sum_{t=0}^{T-1}\tau_td_t\le d_0+b_\varepsilon\sum_{t=0}^{T-1}\tau_t^2.
\end{equation}
\end{lemma}

The step along the negative penalty gradient yields an \(O(\tau_td_t)\) contraction, whereas the quadratic discretization term is \(O(\tau_t^2)\), giving an \(O(\tau)\) residual-floor bound for a constant step size. The coupling of the contraction and forcing terms will motivate the extension in Section~\ref{sec:separate-steps-convergence}. The proofs of Lemmas~\ref{lem:coupled-safeguard} and~\ref{lem:coupled-residual-recursion} are given in Appendices~\ref{app:coupled-safeguard-proof} and~\ref{app:coupled-residual-proof}, respectively.

\subsection{Main Convergence Results}
\label{subsec:coupled-horizon-rates}
\label{subsec:coupled-finite-time}

Combining Lemmas~\ref{lem:coupled-safeguard} and~\ref{lem:coupled-residual-recursion} with the objective-descent and weighted momentum-tracking bounds in Appendices~\ref{app:coupled-objective-descent-proof}--\ref{app:coupled-tracking-accumulation} yields the following finite-time guarantee. Let \(f_\varepsilon^{\inf}:=\min_{\Xbf\in\mathcal S_\varepsilon}f(\Xbf)\) and \(\nu_{\varepsilon,B}:=\sigma\sqrt{n(1+\varepsilon)/B}\). Define
\(W_T:=\sum_{t=0}^{T-1}\tau_t\), \(Q_T:=\sum_{t=0}^{T-1}\tau_t^2\), \(\Delta_f:=f(\Xbf_0)-f_\varepsilon^{\inf}\), and
\begin{equation}
\label{eq:coupled-rate-quantities}
\begin{aligned}
\mathcal R_T^{\mathrm{feas}}
&:=
\frac{d_0+b_\varepsilon Q_T}{a_\varepsilon W_T},
\\
\mathcal R_T^{\mathrm{track}}
&:=
\nu_{\varepsilon,B}
\Bsbra{
\frac{\sum_{t=0}^{T-1}\tau_t\beta^t}{W_T}
+
\sqrt{\frac{1-\beta}{1+\beta}}
}
+
L_{\Psi,\varepsilon}M_\varepsilon
\frac{\beta}{1-\beta}
\frac{Q_T}{W_T},
\\
\mathcal R_T^{\mathrm{stat}}
&:=
\frac{\Delta_f}{W_T}
+
2\mathcal R_T^{\mathrm{track}}
+
\lambda G^{\mathrm{sym}}_\varepsilon\mathcal R_T^{\mathrm{feas}}
+
\frac{L_*M_\varepsilon^2}{2}
\frac{Q_T}{W_T}.
\end{aligned}
\end{equation}

\begin{theorem}[General finite-time guarantee for \algname{}]
\label{thm:coupled-finite-time}
Suppose that Assumptions~\ref{ass:operator-nuclear-smoothness}
and~\ref{ass:stochastic-oracle} hold.
Let \(\Xbf_0\in\mathcal S_\varepsilon\), not necessarily feasible, and
let the deterministic target step sizes satisfy
\eqref{eq:coupled-target-step-condition}.
Let \(R\) be independent of the algorithmic randomness, with \(\mathbb P(R=t)=\tau_t/W_T\) for \(t=0,\ldots,T-1\), and let \(\mathbb E_R\) denote expectation over \(R\) conditional on the iterates.
Then, for every \(T\ge1\),

\begin{equation}
\label{eq:coupled-finite-time-feasibility}
\begin{aligned}
\mathbb E_R[d_R]
&=\frac{\sum_{t=0}^{T-1}\tau_t d_t}{W_T}
\le\mathcal R_T^{\mathrm{feas}}\quad\text{pathwise},\\
\mathbb E\|\Psibf_R\|_*
&=\frac{\sum_{t=0}^{T-1}\tau_t\mathbb E\|\Psibf_t\|_*}{W_T}
\le\mathcal R_T^{\mathrm{stat}}.
\end{aligned}
\end{equation}

\end{theorem}

The proof is given in Appendix~\ref{app:coupled-finite-time-proof}. Choosing a horizon-dependent constant step within each run yields the following rates.

\begin{corollary}[Stochastic-momentum rate]
\label{cor:coupled-stochastic-momentum}
Suppose Assumptions~\ref{ass:operator-nuclear-smoothness}
and~\ref{ass:stochastic-oracle} hold, the mini-batch size \(B\) is fixed independently of \(T\), and
$\Xbf_0\in\operatorname{St}(n,p)$. Choose
\(\eta_t^{\mathrm{tar}}\equiv\tau_T:=\eta_0T^{-3/4}\) and \(\beta=\beta_T:=1-\beta_0T^{-1/2}\),
where $\eta_0>0$ and $\beta_0\in(0,1]$. For all sufficiently large $T$,
$\tau_t=\tau_T$ throughout the run and
\begin{equation}
\frac1T\sum_{t=0}^{T-1}\mathbb E\|\Psibf_t\|_*
=\mathcal O(T^{-1/4}),
\qquad
\frac1T\sum_{t=0}^{T-1}d_t=\mathcal O(T^{-3/4})
\quad\text{pathwise}.
\label{eq:coupled-stochastic-rates}
\end{equation}
\end{corollary}

The stochastic rate is proved in Appendix~\ref{app:coupled-stochastic-rate-proof}.

\begin{corollary}[Deterministic-momentum rate]
\label{cor:coupled-deterministic-no-momentum}
Suppose Assumption~\ref{ass:operator-nuclear-smoothness} holds and exact gradients are used. Let \(\Xbf_0\in\mathcal S_\varepsilon\), fix \(\beta\in[0,1)\) independently of \(T\), and choose \(\eta_t^{\mathrm{tar}}\equiv\tau_T:=\eta_0T^{-1/2}\), where \(\eta_0>0\) is fixed. For all sufficiently large \(T\), \(\tau_t=\tau_T\) throughout the run, and
\begin{equation}
\label{eq:coupled-deterministic-stationarity-explicit}
\begin{aligned}
\frac1T\sum_{t=0}^{T-1}\|\Psibf_t\|_*
\le{}&
\frac{\Delta_f+\lambda G_\varepsilon^{\mathrm{sym}}d_0/a_\varepsilon}
{\eta_0\sqrt T} +\frac{\eta_0}{\sqrt T}
\left(
2L_{\Psi,\varepsilon}M_\varepsilon\frac{\beta}{1-\beta}
+\lambda G_\varepsilon^{\mathrm{sym}}\Gamma_\varepsilon
+\frac{L_*M_\varepsilon^2}{2}
\right).
\end{aligned}
\end{equation}
and
\begin{equation}
\label{eq:coupled-deterministic-feasibility-explicit}
\frac1T\sum_{t=0}^{T-1}d_t
\le
\frac1{\sqrt T}
\left(\frac{d_0}{a_\varepsilon\eta_0}+\Gamma_\varepsilon\eta_0\right).
\end{equation}
Thus,
\begin{equation}
\label{eq:coupled-deterministic-rates}
\frac1T\sum_{t=0}^{T-1}\|\Psibf_t\|_*=\mathcal O(T^{-1/2}),\qquad
\frac1T\sum_{t=0}^{T-1}d_t=\mathcal O(T^{-1/2}).
\end{equation}
In particular, setting \(\beta=0\) recovers the deterministic no-momentum case.
\end{corollary}

The proof is given in Appendix~\ref{app:coupled-deterministic-rate-proof}. Fixed momentum changes the stationarity constant through the tracking error, while the residual bound is unchanged.

\section{\algplus{}: Two-Timescale Step Sizes}
\label{sec:separate-steps-convergence}

To improve feasibility control, we introduce \algplus{}, which retains a fixed normal step while allowing the tangential step to decrease. The analysis uses the assumptions and notation of Section~\ref{sec:preliminaries}.

\paragraph{The \algplus{} update.}

\algplus{}, summarized in Algorithm~\ref{alg:landom-plus}, coincides with Algorithm~\ref{alg:landom} except for the step: it replaces the coupled direction \(\Dbf_t\) and step \(\tau_t\) by a separate tangential step \(\gamma_t\) and a fixed normal step \(\rho\):
\begin{equation}
\Xbf_{t+1}=\Xbf_t-\gamma_t\Omegabf_t\Xbf_t-\rho \Xbf_t\Cbf_t.
\label{eq:separate-steps-update}
\end{equation}

Choose $0<\rho<\bar\rho_\varepsilon$ independently of $T$, with the
admissible range given in \eqref{eq:separate-steps-rho-admissible}, and set
\begin{equation}
\gamma_t=\min\big\{\gamma_t^{\mathrm{tar}},
\gamma_{\varepsilon,\rho}^{\mathrm{safe}}(d_t)\big\},
\label{eq:separate-steps-safeguarded-step}
\end{equation}
where $\gamma_t^{\mathrm{tar}}>0$ is prescribed and the safeguard is
defined in \eqref{eq:separate-steps-safe-map}.

\begin{algorithm}[!hbp]

\caption{\algplus{} with two-timescale step sizes}
\label{alg:landom-plus}
\begin{algorithmic}[1]
\Require $\Xbf_0\in\mathcal S_\varepsilon$, horizon $T$, batch size $B$,
momentum $\beta\in[0,1)$, fixed normal step
$0<\rho<\bar\rho_\varepsilon$, and target step sizes
$\{\gamma_t^{\mathrm{tar}}\}$
\For{$t=0,\ldots,T-1$}
    \State Compute $\Omegabf_t$, $\Cbf_t$, and $d_t$ as in lines~2--\ref{line:coupled-residual} of Algorithm~\ref{alg:landom}
    \State $\gamma_t\gets\min\{\gamma_t^{\mathrm{tar}},
    \gamma_{\varepsilon,\rho}^{\mathrm{safe}}(d_t)\}$
    \State $\Xbf_{t+1}\gets \Xbf_t-\gamma_t\Omegabf_t\Xbf_t-\rho \Xbf_t\Cbf_t$
\EndFor
\State \Return $\Xbf_T$
\end{algorithmic}
\end{algorithm}

\paragraph{Orthogonality-residual dynamics.}

For the update \eqref{eq:separate-steps-update}, define
\begin{equation}
\Ebf_t
:=
\gamma_t\Omegabf_t\Xbf_t+\rho \Xbf_t\Cbf_t,
\qquad
\Xbf_{t+1}=\Xbf_t-\Ebf_t.
\label{eq:separate-steps-displacement}
\end{equation}
Since $\Omegabf_t\in\operatorname{Skew}(n)$, the tangential first-order term cancels,
giving the pathwise identity
\begin{equation}
\Cbf_{t+1}
=
\Cbf_t-2\rho(\Cbf_t+\Cbf_t^2)+\Ebf_t^\top \Ebf_t.
\label{eq:separate-steps-gram-identity}
\end{equation}
In this identity, the coefficient of the first-order normal term depends on \(\rho\), while the tangential step enters only through the quadratic term \(\Ebf_t^\top \Ebf_t\). The following subsection bounds this quadratic term to establish residual contraction.

\subsection{Safeguard and Uniform Residual Contraction}
\label{subsec:separate-steps-feasibility}

Set \(\kappa_\varepsilon:=1+\varepsilon\) and \(\Delta_{\varepsilon,\rho}:=2(1-\varepsilon)-\rho\varepsilon(1+\varepsilon)\), and assume
\begin{equation}
\label{eq:separate-steps-rho-admissible}
0<\rho<\bar\rho_\varepsilon:=\min\left\{\frac12,\frac{2(1-\varepsilon)}{\varepsilon(1+\varepsilon)}\right\}.
\end{equation}
Then \(2\rho<1\) and \(\Delta_{\varepsilon,\rho}>0\). For \(d\in[0,\varepsilon]\), define
\begin{equation}
\label{eq:separate-steps-safe-map}
\gamma_{\varepsilon,\rho}^{\mathrm{safe}}(d):=
\left[\sqrt{\frac{\varepsilon-d+2\rho d(1-d)}{1+d}}-\rho d\right]_+.
\end{equation}
Define the step-size thresholds \(\gamma_\varepsilon^{\mathrm{unif}}(\rho):=\sqrt{2\rho\varepsilon(1-\varepsilon)/(1+\varepsilon)}-\rho\varepsilon\), \(\gamma_{\varepsilon,\rho}^{\mathrm{ctr}}:=\Delta_{\varepsilon,\rho}/(4\kappa_\varepsilon)\), and \(\gamma_{\varepsilon,\rho}^{\mathrm{adm}}:=\min\{\gamma_\varepsilon^{\mathrm{unif}}(\rho),\gamma_{\varepsilon,\rho}^{\mathrm{ctr}}\}\). The admissible cap combines the uniform safeguard threshold with the contraction requirement.

\begin{lemma}[Safeguard and safe-region invariance]
\label{lem:separate-steps-safeguard}
Let \(0<\rho<\bar\rho_\varepsilon\), \(\Xbf\in\mathcal S_\varepsilon\), and \(\Omegabf\in\operatorname{Skew}(n)\) with \(\|\Omegabf\|_{\mathrm{op}}\le1\). Set \(\Cbf:=\Xbf^\top\Xbf-\Ibf_p\) and \(d:=\|\Cbf\|_{\mathrm{op}}\). Then, for every \(0\le\gamma\le\gamma_{\varepsilon,\rho}^{\mathrm{safe}}(d)\),
\[
\Xbf-\gamma\Omegabf\Xbf-\rho\Xbf\Cbf\in\mathcal S_\varepsilon.
\]
Moreover, \(\gamma_{\varepsilon,\rho}^{\mathrm{safe}}(d)\ge\gamma_\varepsilon^{\mathrm{unif}}(\rho)>0\) for all \(d\in[0,\varepsilon]\). Since the compact polar direction satisfies the stated conditions, Algorithm~\ref{alg:landom-plus} initialized at \(\Xbf_0\in\mathcal S_\varepsilon\) remains in \(\mathcal S_\varepsilon\) pathwise.
\end{lemma}

If \(0<\gamma_t^{\mathrm{tar}}\le\gamma_{\varepsilon,\rho}^{\mathrm{adm}}\), then \(\gamma_t=\gamma_t^{\mathrm{tar}}\) pathwise. Define \(a_{\varepsilon,\rho}:=\rho\Delta_{\varepsilon,\rho}/2\) and \(q_{\varepsilon,\rho}:=1-a_{\varepsilon,\rho}\in(0,1)\).

\begin{lemma}[Pathwise residual contraction]
\label{lem:separate-steps-residual-recursion}
Suppose that \(\Xbf_0\in\mathcal S_\varepsilon\), \(0<\rho<\bar\rho_\varepsilon\), and the deterministic target step sizes satisfy \(0<\gamma_t^{\mathrm{tar}}\le\gamma_{\varepsilon,\rho}^{\mathrm{adm}}\). Then
\begin{equation}
\label{eq:separate-steps-residual-recursion}
d_{t+1}\le q_{\varepsilon,\rho}d_t+\kappa_\varepsilon\gamma_t^2\qquad\text{pathwise}.
\end{equation}

If \(\gamma_j\le\bar\gamma\) for all \(j<t\), then

\begin{equation}
\label{eq:separate-steps-residual-floor}
d_t\le q_{\varepsilon,\rho}^{\,t}d_0+c_{\varepsilon,\rho}^{\mathrm{res}}(1-q_{\varepsilon,\rho}^{\,t})\bar\gamma^2,\qquad
c_{\varepsilon,\rho}^{\mathrm{res}}:=\frac{\kappa_\varepsilon}{a_{\varepsilon,\rho}}=\frac{2(1+\varepsilon)}{\rho[2(1-\varepsilon)-\rho\varepsilon(1+\varepsilon)]}.
\end{equation}

\end{lemma}

For constant \(\gamma_t=\gamma_T\), the residual-floor bound separates a geometric initialization transient from an \(\mathcal O(\gamma_T^2/\rho)\) residual floor, compared with \(\mathcal O(\tau_T)\) for \algname{}. The contraction factor \(q_{\varepsilon,\rho}\) is independent of the tangential step size. Summing the recursion also yields a pathwise bound on the full-run average residual; see Appendices~\ref{app:separate-steps-safeguard-proof} and~\ref{app:separate-steps-residual-proof} for details.

\subsection{Main Convergence Results}
\label{subsec:separate-steps-horizon-rates}
\label{subsec:separate-steps-finite-time}

Combining the feasibility bounds above with the objective-descent and weighted momentum-tracking estimates established in Appendices~\ref{app:separate-steps-objective-descent-proof}--\ref{app:separate-steps-tracking-accumulation} yields the following finite-time guarantee. Retain \(\Delta_f=f(\Xbf_0)-f_\varepsilon^{\inf}\) and \(\nu_{\varepsilon,B}\) from Section~\ref{subsec:coupled-horizon-rates}, and set \(A_{\varepsilon,\rho}:=\rho G_\varepsilon^{\mathrm{sym}}+L_*\kappa_\varepsilon\rho^2\varepsilon\). In this section, the step-size sums are
\(G_T:=\sum_{t=0}^{T-1}\gamma_t\), \(Q_T:=\sum_{t=0}^{T-1}\gamma_t^2\), and \(\bar\gamma_T:=\max_{0\le t<T}\gamma_t\); we further define
\begin{equation}
\begin{aligned}
\mathcal Q_T^{\mathrm{feas}}
&:=
\frac{d_0+\kappa_\varepsilon Q_T}{a_{\varepsilon,\rho}T},
\\
\mathcal Q_T^{\mathrm{track}}
&:=
\nu_{\varepsilon,B}
\Bsbra{
\frac{\sum_{t=0}^{T-1}\gamma_t\beta^t}{G_T}
+
\sqrt{\frac{1-\beta}{1+\beta}}
}
\\
&\quad+
L_{\Psi,\varepsilon}\sqrt{\kappa_\varepsilon}
\frac{\beta}{1-\beta}
\Bsbra{
\frac{Q_T}{G_T}
+
\frac{\rho\bar\gamma_T}{a_{\varepsilon,\rho}G_T}
(d_0+\kappa_\varepsilon Q_T)
},
\\
\mathcal Q_T^{\mathrm{stat}}
&:=
\frac{\Delta_f}{G_T}
+2\mathcal Q_T^{\mathrm{track}}
+
\frac{A_{\varepsilon,\rho}}{a_{\varepsilon,\rho}G_T}
(d_0+\kappa_\varepsilon Q_T)
+
L_*\kappa_\varepsilon\frac{Q_T}{G_T}.
\end{aligned}
\label{eq:separate-steps-rate-quantities}
\end{equation}

\begin{theorem}
\label{thm:separate-steps-finite-time}
Suppose Assumptions~\ref{ass:operator-nuclear-smoothness} and
\ref{ass:stochastic-oracle} hold,
$\Xbf_0\in\mathcal S_\varepsilon$,
$0<\rho<\bar\rho_\varepsilon$,
$\beta\in[0,1)$, and
the deterministic target step sizes satisfy $0<\gamma_t^{\mathrm{tar}}\le\gamma_{\varepsilon,\rho}^{\mathrm{adm}}$ for $t=0,\ldots,T-1$. Then

\begin{equation}
\label{eq:separate-steps-finite-time-feasibility}
\frac1T\sum_{t=0}^{T-1}d_t
\le\mathcal Q_T^{\mathrm{feas}}
\quad\text{pathwise},
\qquad
\frac{\sum_{t=0}^{T-1}\gamma_t\mathbb E\|\Psibf_t\|_*}{G_T}
\le\mathcal Q_T^{\mathrm{stat}}.
\end{equation}

\end{theorem}

The proof is given in Appendix~\ref{app:separate-steps-finite-time-proof}. Choosing a horizon-dependent constant tangential step within each run yields the following rates.
Let \(R\sim\operatorname{Unif}\{0,\ldots,T-1\}\) be independent of the algorithmic randomness, and let \(\mathbb E_R\) denote expectation over \(R\) conditional on the iterates.

\begin{corollary}[Stochastic-momentum rate]
\label{cor:separate-steps-stochastic-momentum}
Suppose Assumptions~\ref{ass:operator-nuclear-smoothness} and~\ref{ass:stochastic-oracle} hold and \(\Xbf_0\in\mathcal S_\varepsilon\). Fix \(0<\rho<\bar\rho_\varepsilon\) and the mini-batch size \(B\) independently of \(T\). Choose
$\gamma_t^{\mathrm{tar}}\equiv\gamma_T:=\gamma_0T^{-3/4}$, $\beta=\beta_T:=1-\beta_0T^{-1/2}$,
where $\gamma_0>0$ and $\beta_0\in(0,1]$.
Then, for all sufficiently large $T$,

\begin{equation}
\label{eq:separate-steps-stochastic-rates}
\mathbb E\|\Psibf_R\|_*=\mathcal O(T^{-1/4}),\qquad
\mathbb E_R[d_R]=\mathcal O(d_0/T+T^{-3/2})\quad\text{pathwise}.
\end{equation}

If $d_0=0$, the stronger pathwise feasibility guarantees hold:

\begin{equation}
\label{eq:separate-steps-stochastic-pathwise-feasibility}
\max_{0\le t\le T}d_t=\mathcal O(T^{-3/2}).
\end{equation}

\end{corollary}

The stochastic rate is proved in Appendix~\ref{app:separate-steps-stochastic-rate-proof}.

\begin{corollary}[Deterministic-momentum rate]
\label{cor:separate-steps-deterministic-no-momentum}
Suppose Assumption~\ref{ass:operator-nuclear-smoothness} holds and exact gradients are used. Let \(\Xbf_0\in\mathcal S_\varepsilon\), fix \(0<\rho<\bar\rho_\varepsilon\) and \(\beta\in[0,1)\) independently of \(T\), and choose \(\gamma_t^{\mathrm{tar}}\equiv\gamma_T:=\gamma_0T^{-1/2}\), where \(\gamma_0>0\) is fixed. For all sufficiently large \(T\), \(\gamma_t=\gamma_T\) throughout the run, and
\begin{equation}
\label{eq:separate-steps-deterministic-stationarity-explicit}
\begin{aligned}
\mathbb E_R\|\Psibf_R\|_*
\le{}&
\frac{\Delta_f+A_{\varepsilon,\rho}d_0/a_{\varepsilon,\rho}}
{\gamma_0\sqrt T}
+\frac{\gamma_0}{\sqrt T}
\left(
2L_{\Psi,\varepsilon}\sqrt{\kappa_\varepsilon}\frac{\beta}{1-\beta}
+\frac{A_{\varepsilon,\rho}\kappa_\varepsilon}{a_{\varepsilon,\rho}}
+L_*\kappa_\varepsilon
\right)\\
&+\frac{2L_{\Psi,\varepsilon}\sqrt{\kappa_\varepsilon}\beta\rho(d_0+\kappa_\varepsilon\gamma_0^2)}
{(1-\beta)a_{\varepsilon,\rho}T}.
\end{aligned}
\end{equation}
and
\begin{equation}
\label{eq:separate-steps-deterministic-feasibility-explicit}
\frac1T\sum_{t=0}^{T-1}d_t\le\frac{d_0+\kappa_\varepsilon\gamma_0^2}{a_{\varepsilon,\rho}T}.
\end{equation}
Consequently,
\begin{equation}
\label{eq:separate-steps-deterministic-rates}
\mathbb E_R\|\Psibf_R\|_*=\mathcal O(T^{-1/2}),\qquad
\mathbb E_R[d_R]=\mathcal O(T^{-1})\quad\text{pathwise}.
\end{equation}
If \(d_0=0\), then
\begin{equation}
\label{eq:separate-steps-deterministic-pathwise-feasibility}
\max_{0\le t\le T}d_t
\le\frac{c_{\varepsilon,\rho}^{\mathrm{res}}\gamma_0^2}{T}.
\end{equation}
In particular, setting \(\beta=0\) recovers the deterministic no-momentum case.
\end{corollary}

The proof is given in Appendix~\ref{app:separate-steps-deterministic-rate-proof}. Fixed momentum preserves the stationarity and feasibility orders; its additional contribution appears only in the stationarity bound.

\begin{remark}[Burn-in period]
\label{rem:burn-in-period}
Starting from an infeasible point, an initial burn-in period may be needed for the residual to reach \(O(\tau_T)\) for \algname{} and \(O(\gamma_T^2/\rho)\) for \algplus{}. The sufficient burn-in lengths are \(O(T^{3/4}\log T)\) and \(O(T^{1/2}\log T)\) for \algname{} in the stochastic and deterministic settings, respectively, but only \(O(\log T)\) for \algplus{} in both settings. See Appendices~\ref{app:coupled-tail-guarantees-proof} and~\ref{app:separate-steps-tail-guarantees-proof} for details.
\end{remark}

\section{Experiments}
\label{sec:experiments}

We evaluate the proposed methods through a series of experiments on synthetic benchmarks and orthogonality-constrained problems arising in machine learning and deep learning.
The proposed methods permit intermediate iterates outside the Stiefel
manifold, for which raw objective values may be affected by constraint
violation. 
Whenever a reference value is available, we evaluate objective quality using the objective gap after projection,
\(g_{\mathrm{proj}}(\Xbf)=f(\Pi_{\mathrm{St}}(\Xbf))-f_{\mathrm{ref}}\), where \(\Pi_{\mathrm{St}}(\Xbf)=\Xbf(\Xbf^\top\Xbf)^{-1/2}\).
Here \(f_{\mathrm{ref}}\) denotes the optimal value for linear Stiefel optimization and PCA, and the empirical reference for ICA. Projection is used only for objective evaluation and does not alter the iterates. Feasibility is measured by \(d(\Xbf)\), except for the neural-network residual defined below, and stationarity by \(\|\Psi(\Xbf)\|_*\) (Section~\ref{sec:preliminaries}).
Unless a comparison requires matched hyperparameters, each method is
tuned independently under the same tuning budget, and the
best-performing configuration within its search range is used.

\subsection{Linear Optimization over the Stiefel Manifold}
\label{subsec:exp1-linear-stiefel}

We first consider the deterministic linear optimization problem over the Stiefel manifold
\(\min_{\Xbf\in\operatorname{St}(n,p)}f(\Xbf)=-\operatorname{tr}(\Mbf^\top\Xbf)\).
If \(\Mbf=\Ubf\Sigmabf\Vbf^\top\) is a thin SVD of \(\Mbf\), then \(\Xbf^\star=\Ubf\Vbf^\top\) is an optimal feasible solution, with \(f_{\mathrm{ref}}=f^\star=-\|\Mbf\|_*\). We set \((n,p)=(200,50)\) and consider three
singular-value profiles: easy,
\(\sigma_i=10^{-(i-1)/(p-1)}\); ill-conditioned,
\(\sigma_i=10^{-4(i-1)/(p-1)}\); and block, with \(\sigma_i=1\) for
\(i\le p/2\) and \(\sigma_i=10^{-3}\) for \(i>p/2\). The closed-form
optimum enables direct evaluation of objective quality, feasibility,
and stationarity, as well as the residual-scaling behavior across
different spectral profiles.

\begin{figure}[tb]
\centering
\includegraphics[width=0.96\linewidth]{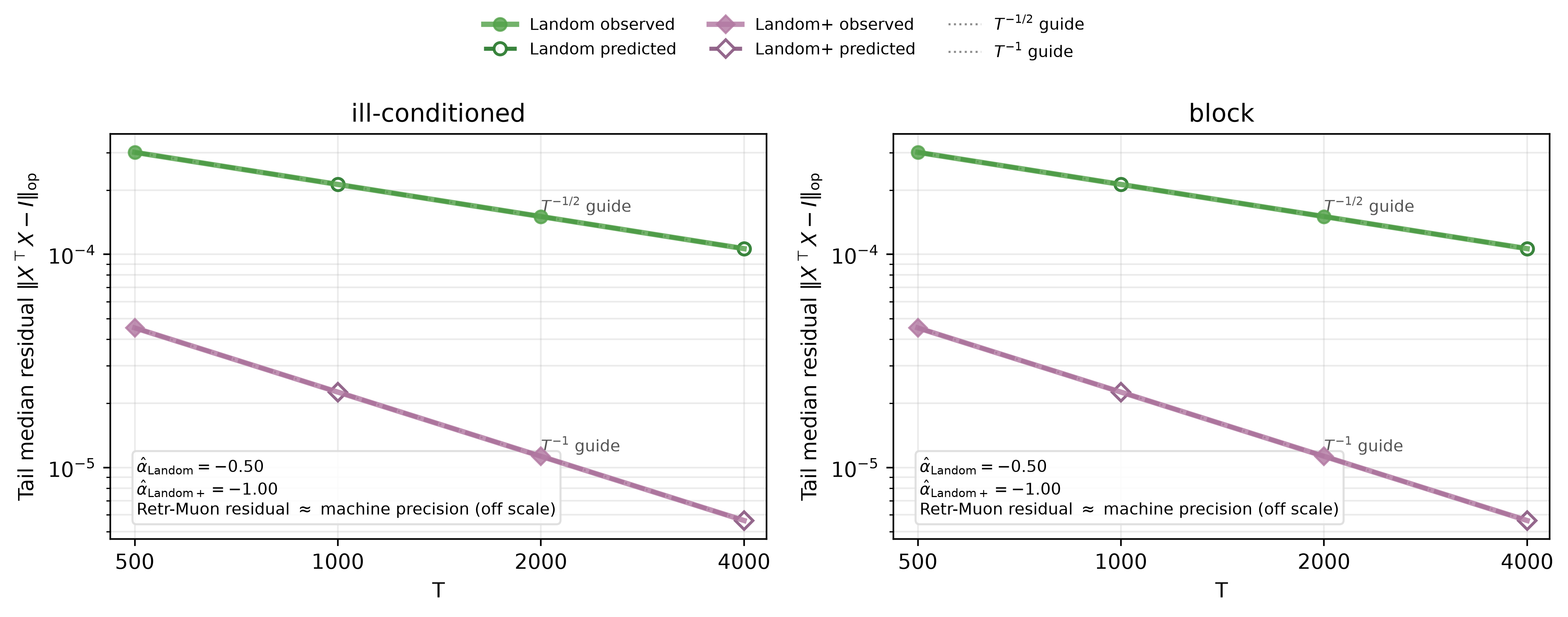}
\par\smallskip
\includegraphics[width=0.96\linewidth]{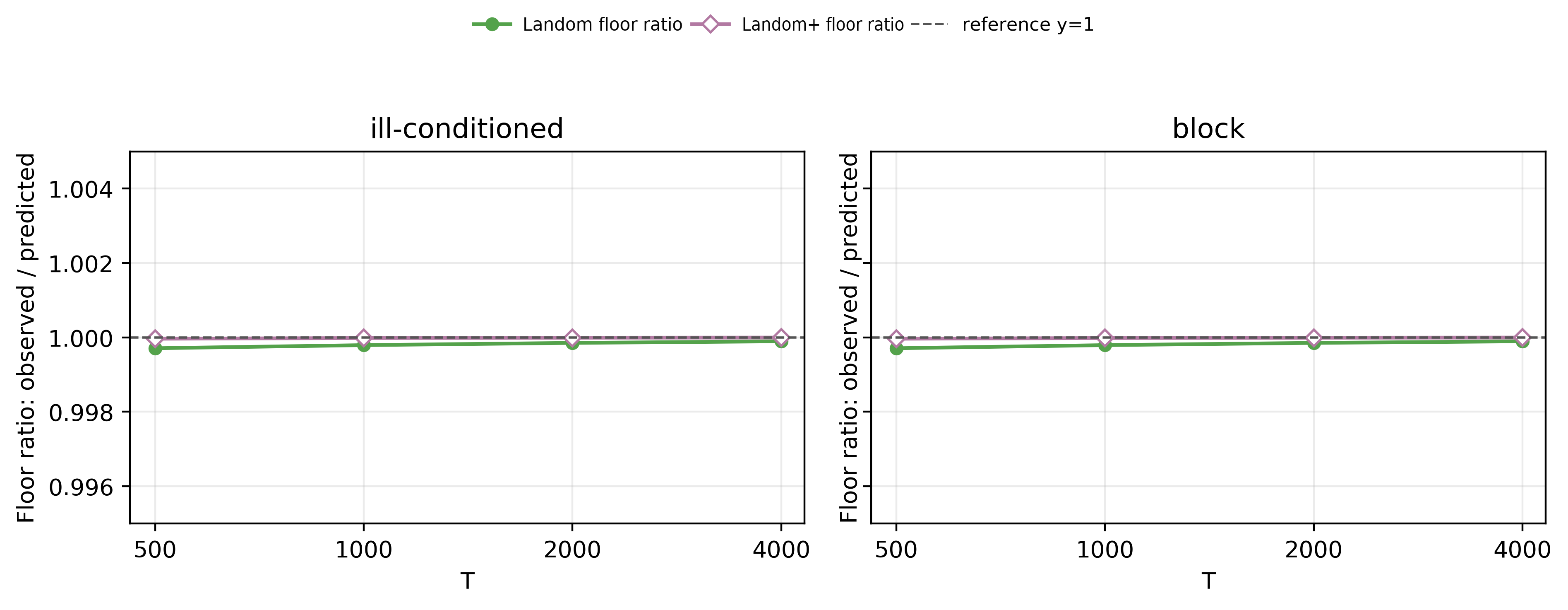}
\caption{\textbf{Residual scaling.} Top: tail median of
\(\|\Xbf_t^\top \Xbf_t-\Ibf_p\|_{\mathrm{op}}\) versus \(T\). Bottom:
observed-to-predicted residual-floor ratio.}
\label{fig:exp1-residual-scaling}
\end{figure}

We compare Landing-GD, a retraction-based polar-factor method (Retr-Muon in the figures), \algname{}, and \algplus{}. We use two step-size
settings. To examine the predicted residual scaling, for each
\(T\in\{500,1000,2000,4000\}\), the horizon experiment uses the constant
target tangential step size \(0.1344T^{-1/2}\) for all three polar-factor methods throughout the run,
while \algplus\ uses the fixed normal step size \(\rho=0.4\).
For the practical optimization comparison, we instead use the
diminishing target tangential schedule
\(1.3389\times10^{-3}(1+t/750)^{-0.75}\) for all three polar-factor methods, while \algplus\ uses \(\rho=0.498\).
Thus, their differences are not attributable to different tangential
step sizes. For both \algname{} and \algplus{}, we set the
safe-region parameter to \(\varepsilon=0.5\).
Figures~\ref{fig:exp1-optimization-dynamics}--%
\ref{fig:exp1-budget-sweep} use the same five seeds.
Figure~\ref{fig:exp1-optimization-dynamics} shows the mean trajectories
with one-standard-deviation bands, whereas
Figure~\ref{fig:exp1-budget-sweep} first takes the median over the final
\(20\%\) of the iterates for each seed and then aggregates across seeds.

\begin{figure}[tb]
\centering
\includegraphics[width=0.93\linewidth]{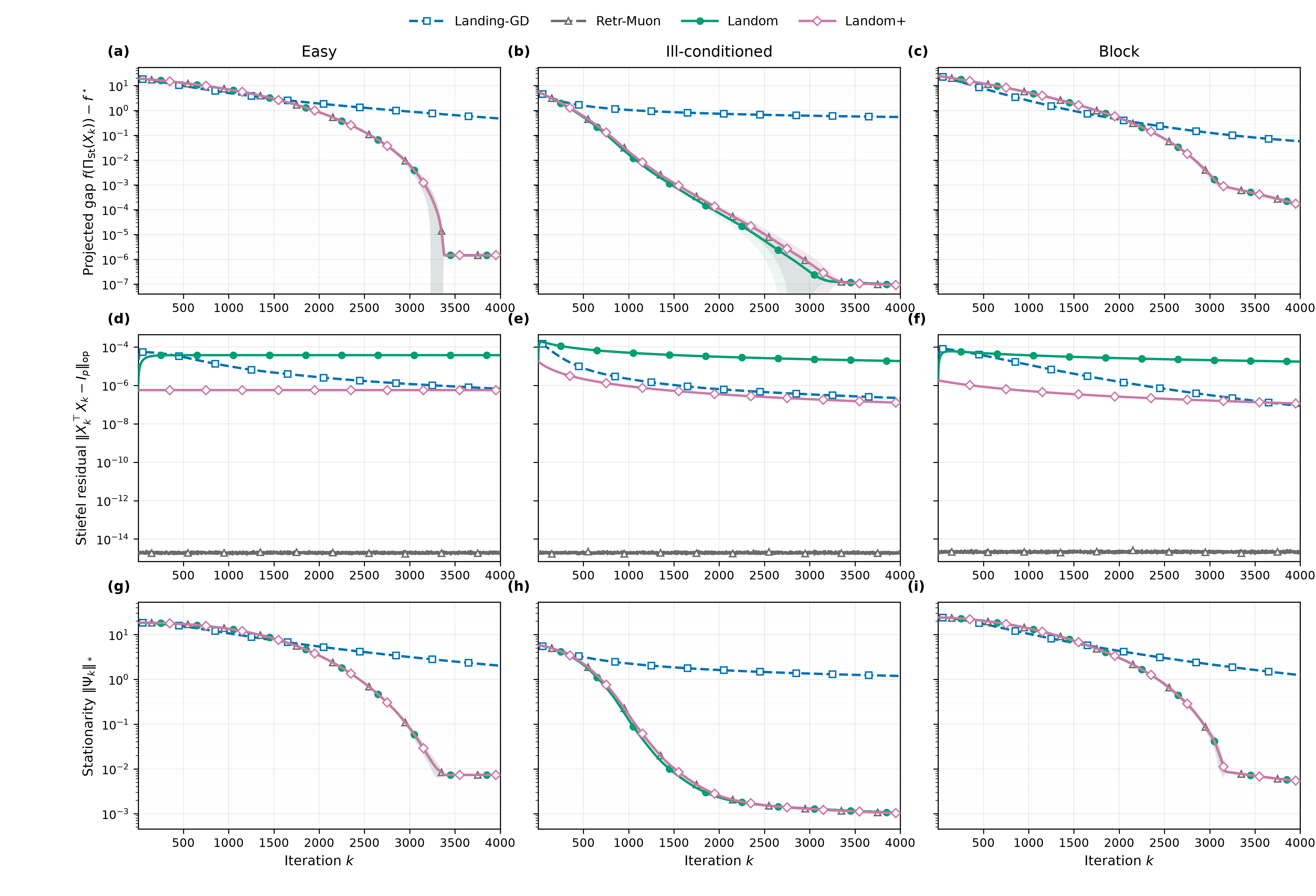}
\caption{\textbf{Optimization dynamics.} Objective gap after projection,
orthogonality residual, and nuclear-norm stationarity at \(T=4000\) for the
three spectral profiles. Curves show means over five seeds and shaded
regions indicate one standard deviation.}
\label{fig:exp1-optimization-dynamics}
\end{figure}

\paragraph{Residual scaling.}
For the horizon experiment, let
\(\mathcal I_T=\{\lceil0.8T\rceil,\ldots,T\}\) and measure the tail
residual by \(\operatorname{median}_{t\in\mathcal I_T}d(\Xbf_t)\).
Because the common target tangential step size scales as \(T^{-1/2}\),
Equation~\eqref{eq:coupled-residual-recursion} predicts an
\(O(T^{-1/2})\) residual floor for \algname{}, whereas
Equation~\eqref{eq:separate-steps-residual-floor} predicts an
\(O(T^{-1})\) residual floor for \algplus{} with fixed \(\rho\).
Figure~\ref{fig:exp1-residual-scaling} gives fitted log-log slopes of
approximately \(-0.50\) and \(-1.00\), respectively, in agreement
with these predictions, while the observed residuals remain close to
the corresponding predicted floors. The precise fitted slopes and
safeguard statistics are given in
Appendix~\ref{app:exp1-residual-scaling}.

\paragraph{Optimization performance.}
Figure~\ref{fig:exp1-optimization-dynamics} shows the optimization
trajectories at \(T=4000\). Across all three spectral profiles, the
polar-factor methods achieve substantially smaller objective gaps after
projection and nuclear-norm stationarity measures than the Landing gradient
baseline, demonstrating the advantage of the polar-factor directions
in objective optimization and stationarity. The retraction-based polar-factor baseline, \algname, and \algplus\ exhibit nearly
identical objective-gap and stationarity trajectories, while
\algplus\ attains a substantially smaller Stiefel residual
than \algname. The retraction-based method remains essentially
feasible.

Figure~\ref{fig:exp1-budget-sweep} shows the same overall behavior across
iteration budgets: the advantage of the polar-factor methods in the
objective gap after projection and in stationarity becomes clearer as
the budget increases,
while \algplus\ preserves this optimization behavior and improves
feasibility relative to \algname.

\begin{figure}[tb]
\centering
\includegraphics[width=0.93\linewidth]{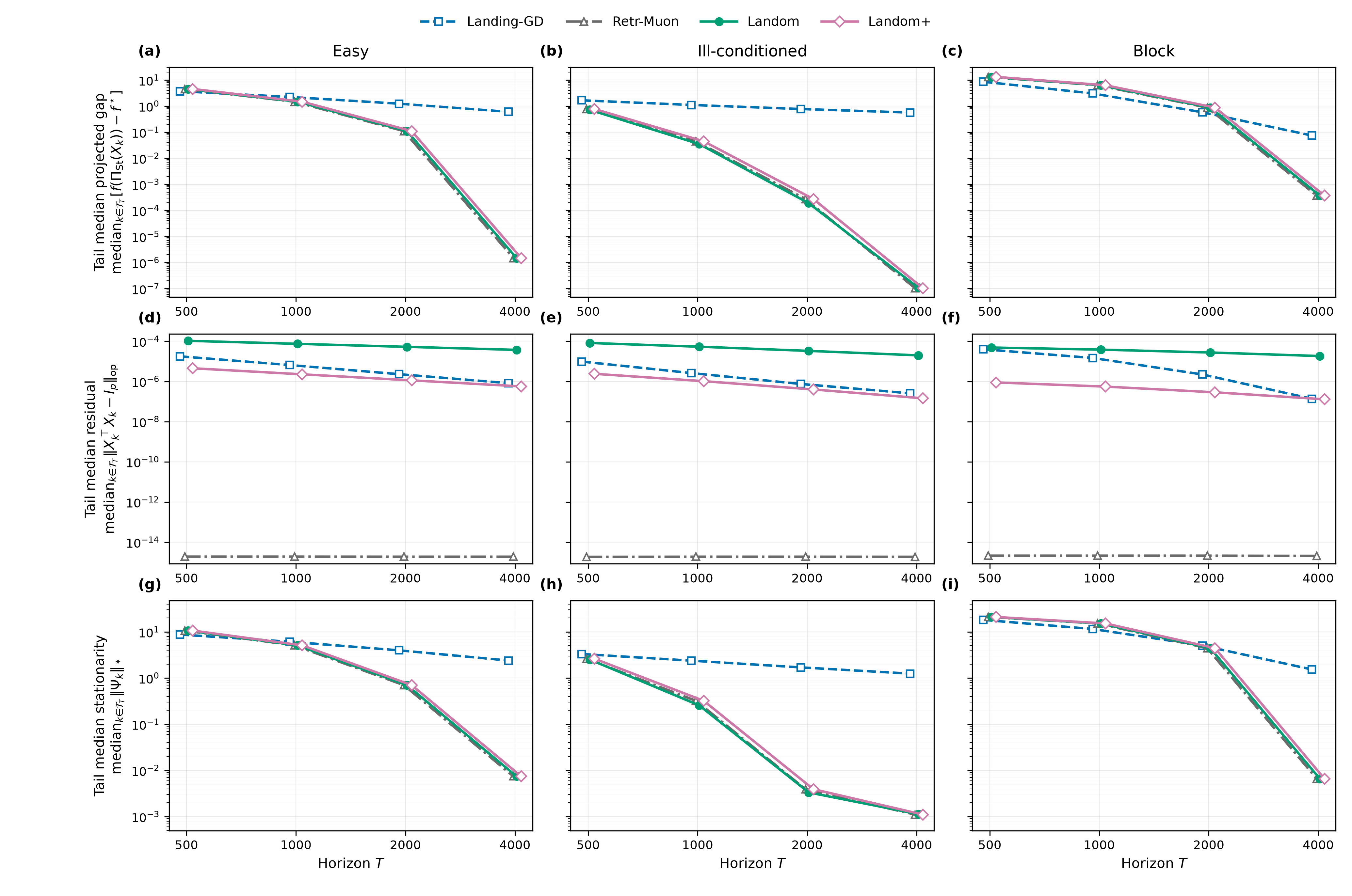}
\caption{\textbf{Budget dependence.} Tail median of the objective gap
after projection, orthogonality residual, and nuclear-norm stationarity for
\(T\in\{500,1000,2000,4000\}\) on the three spectral profiles. Markers
show means over five seeds and bands indicate one standard deviation.}
\label{fig:exp1-budget-sweep}
\end{figure}

\subsection{Stochastic Online PCA}
\label{subsec:exp2-online-pca}

We next consider stochastic online PCA, a standard benchmark for
optimization under orthogonality constraints \citep{ablin2024infeasible}.
Given a data matrix \(\Abf\in\mathbb{R}^{N\times n}\) and
\(\widehat{\Sigmabf}=N^{-1}\Abf^\top \Abf\), we solve
\(\min_{\Xbf\in\operatorname{St}(n,p)}f(\Xbf)=-\frac12\operatorname{tr}(\Xbf^\top\widehat{\Sigmabf}\Xbf)=-\frac1{2N}\|\Abf\Xbf\|_F^2\).
The optimal column space is the leading \(p\)-dimensional eigenspace of \(\widehat{\Sigmabf}\); any orthonormal basis provides an optimal feasible reference point attaining the same value \(f_{\mathrm{ref}}=f^\star=-\frac12\sum_{i=1}^p\lambda_i(\widehat{\Sigmabf})\). We use a
controlled-spectrum synthetic instance with \(N=15000\), \(n=2000\),
and \(p=100\), focusing on a top-\(p\) ill-conditioned spectrum whose
leading components have strongly nonuniform scales while retaining a
separated \(p\)-th eigengap. Stochastic gradients are computed with
mini-batches of size \(B=128\).

We compare Landing-SGD, Riemannian stochastic gradient descent with QR retractions (RSGD-QR), \algname{}, and \algplus{}. The two proposed methods use the same polar-factor construction but differ in whether the normal step is coupled to the tangential step. Since their per-iteration costs differ, we use a common 8-second clean GPU-time budget, covering mini-batch gradient evaluation and optimizer updates. All methods follow the same hardware and timing convention; projected-objective evaluation and diagnostics are excluded from timing. Results are aggregated over five random seeds, with pointwise means and interquartile-range bands.

\begin{figure}[!htbp]
\centering
\includegraphics[width=0.98\linewidth]{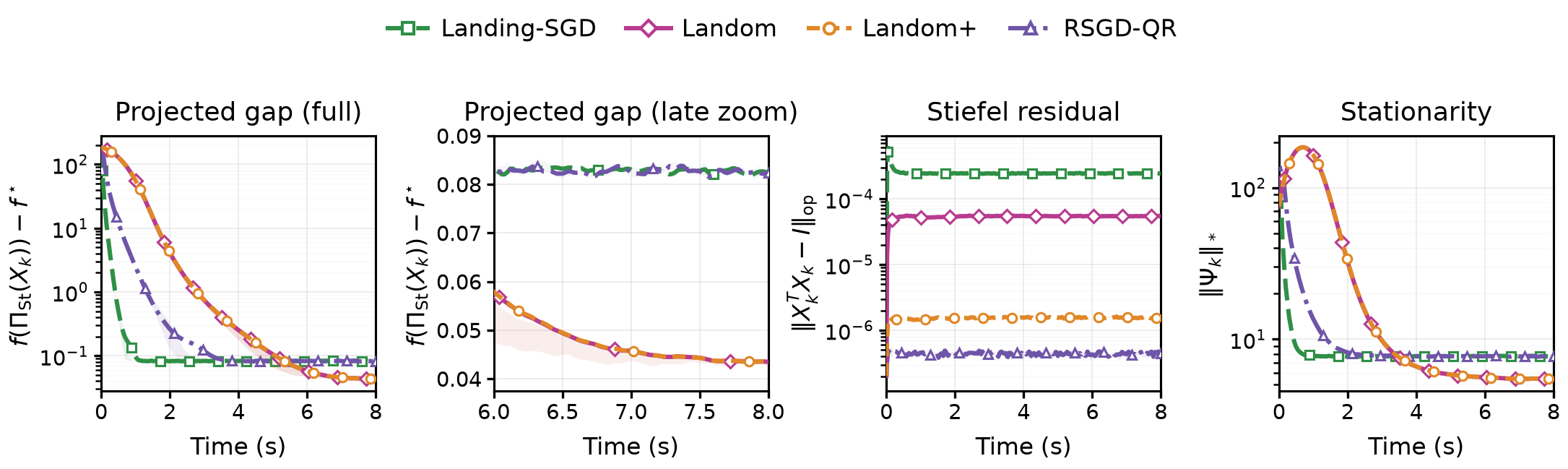}
\caption{\textbf{Equal-time online PCA.} Objective gap after projection (full and
late-time), Stiefel residual, and nuclear-norm stationarity on the
top-\(p\) ill-conditioned spectrum.}
\label{fig:exp2-top-p-main}
\end{figure}

Figure~\ref{fig:exp2-top-p-main} shows that, by the end of the common 8-second clean GPU-time budget, both \algname{} and \algplus{} attain smaller projected objective gaps and nuclear-norm stationarity measures than Landing-SGD and RSGD-QR. Their objective-gap and stationarity trajectories are nearly identical, whereas \algplus{} yields a substantially smaller Stiefel residual than \algname{}. RSGD-QR remains nearly feasible due to the QR retraction. These results support the benefit of polar-factor directions for objective optimization and of decoupled step sizes for feasibility; the latter is consistent with Equation~\eqref{eq:separate-steps-residual-recursion}. Additional results for the easy spectrum and the sensitivity to the normal step size are given in Appendix~\ref{app:exp2-online-pca}.

\FloatBarrier
\subsection{Real-Data Independent Component Analysis}
\label{subsec:exp3-eeg-ica}

We consider independent component analysis (ICA) on real
electroencephalography (EEG) data following the experimental setting of
\citet{xiong2026second}. Given a whitened data matrix
\(\Wbf\in\mathbb{R}^{N\times d}\), we solve
\(\min_{\Xbf\in\operatorname{St}(d,d)}f(\Xbf)=-\frac1N\sum_{i=1}^N\sum_{j=1}^d\log\cosh((\Wbf\Xbf)_{ij})\).
We use the MNE-Python sample dataset~\citep{gramfort2013mne}\footnote{\url{https://mne.tools/dev/documentation/datasets.html\#sample}} and follow the preprocessing of
\citet{xiong2026second}: the first \(60\) seconds of EEG signals are
high-pass filtered at \(1\) Hz, resampled to \(100\) Hz, centered, and
whitened using truncated SVD. This gives \(N=6000\), \(d=60\), and
\(\Xbf\in\operatorname{St}(60,60)\).

The main experiment compares Landing-GD, \algname{}, and \algplus{} using full-batch gradients over \(T=4000\) iterations. We report pointwise medians and interquartile ranges over six seeds. Since the global optimum is unknown, we use the shared empirical reference \(f_{\mathrm{ref}}:=\min_{\Xbf\in\mathcal R_{\mathrm{fb}}}f(\Pi_{\mathrm{St}}(\Xbf))\), where \(\mathcal R_{\mathrm{fb}}\) contains recorded iterates from the formal full-batch reference pool for Figure~\ref{fig:exp3-ica-fullbatch-main}. This is a feasible reference value, not a certified global optimum. All methods use preselected diminishing step sizes and the same initialization within each seed.

\begin{figure}[!htbp]
\centering
\includegraphics[width=0.98\linewidth]{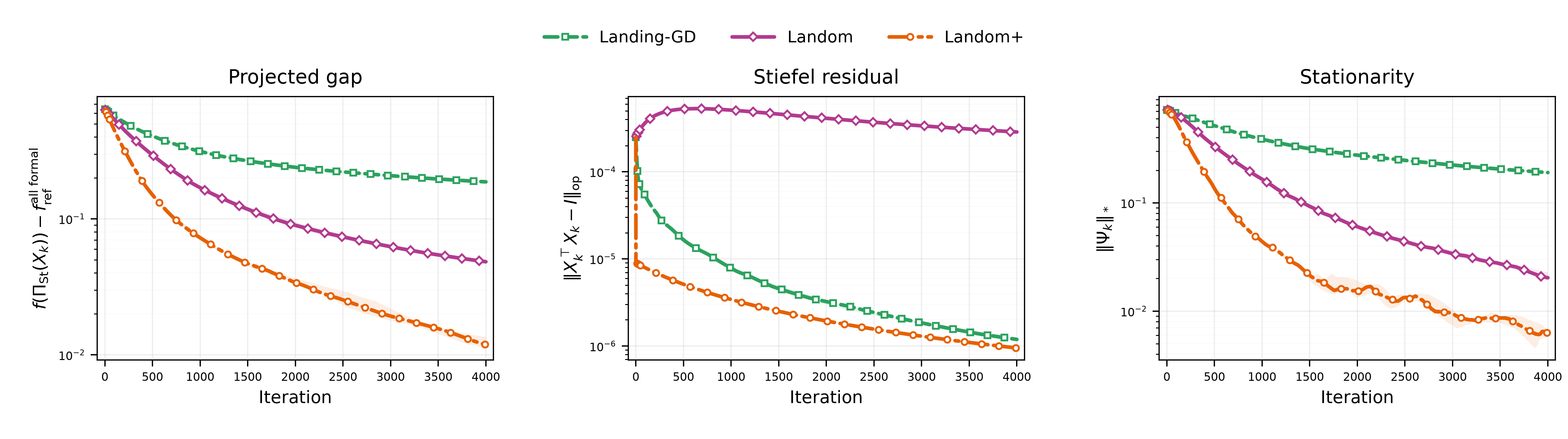}
\caption{\textbf{Real-data ICA.} Objective gap after projection, Stiefel
residual, and nuclear-norm stationarity in the full-batch setting.
Curves show pointwise medians over six seeds and shaded regions indicate
interquartile ranges.}
\label{fig:exp3-ica-fullbatch-main}
\end{figure}

Figure~\ref{fig:exp3-ica-fullbatch-main} shows that both \algname{} and \algplus{} attain smaller projected objective gaps and nuclear-norm stationarity measures than Landing-GD on this real-data, nonquadratic problem. At \(T=4000\), \algplus{} further improves both metrics over \algname{} and reduces its Stiefel residual by more than two orders of magnitude. These results support the effectiveness of polar-factor directions for objective optimization and of decoupled step sizes for feasibility control.

\subsection{Neural Networks with Orthogonality Constraints}
\label{subsec:exp4-neural-network}

Finally, we consider neural-network training with orthogonality
constraints on CIFAR-10~\citep{krizhevsky2009learning}. We use a three-layer bias-free multilayer perceptron (MLP) with dimensions \(3072\!\to\!128\!\to\!64\!\to\!10\) and weight matrices constrained to the corresponding Stiefel manifolds. The network is
trained for \(5\) epochs with batch size \(1024\), and the results are
aggregated over ten random seeds. For the three constrained layers, let \(\Xbf_\ell\in\mathbb R^{n_\ell\times p_\ell}\) denote the transpose of the stored weight matrix, with dimensions \(3072\times128\), \(128\times64\), and \(64\times10\), respectively. We measure orthogonality by the aggregate squared Frobenius residual \(D_{\mathrm{NN}}:=\sum_{\ell=1}^{3}\|\Xbf_\ell^\top\Xbf_\ell-\Ibf_{p_\ell}\|_F^2\).

We compare Landing-SGD and \algplus{} with Manifold Muon \citep{bernstein2025manifolds} and its ADMM variant (Muon-ADMM) \citep{buchanan2025mmuonadmm}. Manifold Muon uses an iterative dual solver for the tangent-space-constrained spectral steepest-descent subproblem, whereas Muon-ADMM employs ADMM for the corresponding inner problem. In contrast, \algplus{} requires no such inner solve and approximates the compact polar factor using Polar Express \citep{amsel2026polar}, avoiding an exact SVD at each update.

Figure~\ref{fig:exp4-nn} shows that \algplus{} achieves training loss and test accuracy comparable to Manifold Muon and Muon-ADMM, with a smaller aggregate orthogonality residual, while improving both metrics over Landing-SGD. Table~\ref{tab:exp4-training-time} reports steady-state train-only epoch times relative to Landing-SGD: \(1.13\times\) for \algplus{}, \(2.99\times\) for Manifold Muon, and \(4.77\times\) for Muon-ADMM.

\begin{figure}[ht]
\centering
\includegraphics[width=0.98\linewidth]{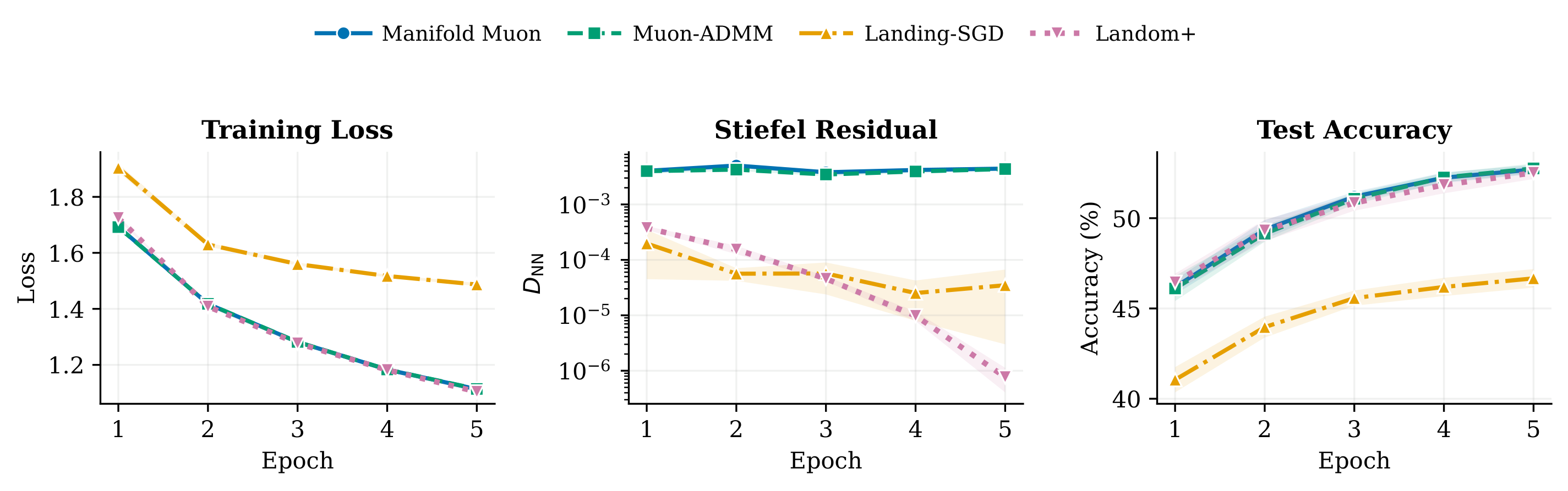}
\caption{\textbf{Orthogonally constrained neural network.} Training loss, aggregate squared Frobenius orthogonality residual \(D_{\mathrm{NN}}\), and test accuracy on CIFAR-10.}
\label{fig:exp4-nn}
\end{figure}

\begin{table}[ht]
\centering
\small
\setlength{\tabcolsep}{6pt}
\begin{tabular}{@{}llcc@{}}
\hline
Method & Inner solve & Time / epoch (s) & Relative time \\
\hline
Landing-SGD & None & \(5.25\pm0.14\) & \(1.00\times\) \\
\algplus & None & \(5.96\pm0.02\) & \(1.13\times\) \\
Manifold Muon & Dual ascent & \(15.72\pm0.19\) & \(2.99\times\) \\
Muon-ADMM & ADMM & \(25.08\pm0.10\) & \(4.77\times\) \\
\hline
\end{tabular}
\caption{\textbf{Training time.} Steady-state train-only wall-clock
time per epoch.}
\label{tab:exp4-training-time}
\vspace{1mm}
\parbox{0.96\linewidth}{\footnotesize Mean \(\pm\) standard deviation
over three timing repeats on a single NVIDIA GeForce RTX 4090. Epoch 1
is excluded as a cold start; evaluation and diagnostic computations are
excluded from the timing.}
\end{table}

\FloatBarrier
\section{Conclusion}
\label{sec:conclusion}

We proposed \algname{} and \algplus{} for retraction-free optimization over the Stiefel manifold. The methods attain \(O(T^{-1/2})\) and \(O(T^{-1/4})\) averaged nuclear-norm stationarity with exact gradients and stochastic momentum, respectively, with the stochastic bound holding in expectation. By separating the tangential and normal step sizes, \algplus{} improves the residual-floor bound from \(O(\tau_T)\) to \(O(\gamma_T^2/\rho)\) and reduces the sufficient burn-in to \(O(\log T)\), while preserving the stationarity rates. Experiments support the predicted feasibility behavior and demonstrate competitive optimization performance across several applications.

Future work includes extending this retraction-free design to other matrix manifolds and structured matrix constraints, with update directions and feasibility-restoration mechanisms adapted to their geometry. Another direction is to scale \algname{} and \algplus{} to large language model training and investigate efficient implementations and empirical behavior on large Transformer matrix parameters.

%% file: appendix.tex
\appendix
\section{Shared Geometric and Stochastic Estimates}
\label{app:shared-proofs}

This appendix proves the geometric identities and basic estimates from Section~\ref{sec:preliminaries} and the polar-alignment and momentum-tracking bounds for \algname{} and \algplus{}.

\subsection{Proof of Lemma~\ref{lem:safe-region-bounds}}
\label{app:shared-safe-region}

For any \(\Xbf\in\mathcal S_\varepsilon\), the condition \(\|\Xbf^\top\Xbf-\Ibf_p\|_{\mathrm{op}}\le\varepsilon\) yields \((1-\varepsilon)\Ibf_p\preceq\Xbf^\top\Xbf\preceq(1+\varepsilon)\Ibf_p\), which implies the stated singular-value and norm bounds, ensures full column rank since \(\varepsilon<1\), and establishes compactness of \(\mathcal S_\varepsilon\) by closedness and boundedness.

\subsection{Normal Space and Relative-Gradient Identities}
\label{app:shared-level-set-identities}

Fix \(\Xbf\in\mathcal S_\varepsilon\) and define \(F_{\Xbf}(\Ybf):=\Ybf^\top\Ybf-\Xbf^\top\Xbf\). Its derivative and adjoint satisfy \(\Df F_{\Xbf}(\Xbf)[\Zbf]=\Xbf^\top\Zbf+\Zbf^\top\Xbf\) and \(\Df F_{\Xbf}(\Xbf)^*[\Sbf]=2\Xbf\Sbf\) for every \(\Sbf\in\mathbb S^p\). Since \(\Pbf:=\Xbf^\top\Xbf\succ0\), choosing \(\Zbf=\frac12\Xbf\Pbf^{-1}\Sbf\) gives \(\Df F_{\Xbf}(\Xbf)[\Zbf]=\Sbf\). Hence, \(\Df F_{\Xbf}(\Xbf)\) is surjective and
\[
N_{\Xbf}\mathcal M_{\Xbf}=\operatorname{range}\Df F_{\Xbf}(\Xbf)^*=\{\Xbf\Sbf:\Sbf\in\mathbb S^p\}.
\]
Since \(C(\Xbf)\in\mathbb S^p\), the normalizing direction \(\Xbf C(\Xbf)\) is orthogonal to \(\Abf\Xbf\) for every \(\Abf\in\operatorname{Skew}(n)\). Moreover, trace cyclicity and the orthogonality of symmetric and skew-symmetric matrices give
\[
\Df f(\Xbf)[\Abf\Xbf]=\langle\nabla f(\Xbf)\Xbf^\top,\Abf\rangle=\langle\Psi(\Xbf),\Abf\rangle.
\]
Now let \(\Xbf\in\operatorname{St}(n,p)\), so that \(\grad f(\Xbf)=\Psi(\Xbf)\Xbf\). The implication \(\Psi(\Xbf)=0\Rightarrow\grad f(\Xbf)=0\) is immediate. Conversely, if \(\Psi(\Xbf)\Xbf=0\) and \(\Gbf:=\nabla f(\Xbf)\), then \(\Gbf=\Xbf\Gbf^\top\Xbf\). Setting \(\Kbf:=\Gbf^\top\Xbf\) gives \(\Gbf=\Xbf\Kbf\) and \(\Kbf^\top=\Xbf^\top\Gbf=\Kbf\). Consequently, \(\Psi(\Xbf)=\frac12\Xbf(\Kbf-\Kbf^\top)\Xbf^\top=0\). Therefore,
\[
\Psi(\Xbf)=0\iff\grad f(\Xbf)=0,\qquad\|\grad f(\Xbf)\|_F=\|\Psi(\Xbf)\Xbf\|_F\le\|\Psi(\Xbf)\|_F\le\|\Psi(\Xbf)\|_*.
\]

\subsection{Proof of Lemma~\ref{lem:matrix-sign-properties}}
\label{app:shared-msign}

The case \(\Hbf=0\) is immediate. For \(\Hbf\ne0\), let \(\Hbf=\Ubf_r\Sigmabf_r\Vbf_r^\top\) be a compact SVD and set \(\Omegabf:=\Ubf_r\Vbf_r^\top\). Then \(\|\Omegabf\|_{\mathrm{op}}=1\) and \(\langle\Hbf,\Omegabf\rangle=\operatorname{tr}(\Sigmabf_r)=\|\Hbf\|_*\), so nuclear--operator duality establishes the maximization property.

For a positive semidefinite matrix \(\Abf\), write \(\Abf^{\dagger/2}:=(\Abf^\dagger)^{1/2}\), where \(\Abf^\dagger\) denotes its Moore--Penrose pseudoinverse. If \(\Hbf\in\operatorname{Skew}(n)\), then \(\Omegabf=\Hbf(\Hbf^\top\Hbf)^{\dagger/2}\). Since \(\Hbf^\top\Hbf=-\Hbf^2\), the matrix \(\Hbf\) commutes with the symmetric spectral function \((\Hbf^\top\Hbf)^{\dagger/2}\), and hence \(\Omegabf^\top=-\Omegabf\). Thus, \(\Omegabf\) also maximizes the inner product over the skew-symmetric operator-norm unit ball.

\subsection{Proof of Proposition~\ref{prop:shared-basic-estimates}}
\label{app:shared-basic-estimates-proof}

Let \(\Xbf,\Ybf\in\mathcal S_\varepsilon\) and \(\Dbf:=\Ybf-\Xbf\).
For part~\ref{prop:shared-basic-estimates-descent}, convexity of
\(\mathcal U_\varepsilon\) ensures
\(\Xbf+s\Dbf\in\mathcal U_\varepsilon\) for \(s\in[0,1]\).
Hence, the fundamental theorem of calculus,
Assumption~\ref{ass:operator-nuclear-smoothness}, and
nuclear--operator duality give
\[
\begin{aligned}
f(\Ybf)-f(\Xbf)
={}&
\langle\nabla f(\Xbf),\Dbf\rangle
+
\int_0^1
\langle\nabla f(\Xbf+s\Dbf)-\nabla f(\Xbf),\Dbf\rangle\,\textrm{d}s
\\
\le{}&
\langle\nabla f(\Xbf),\Dbf\rangle
+
\frac{L_*}{2}\|\Dbf\|_{\mathrm{op}}^2.
\end{aligned}
\]

For part~\ref{prop:shared-basic-estimates-psi-lipschitz}, linearity and contractivity of
the \(\skewop\) operator, together with
\(\|\Abf \Bbf\|_*\le\|\Abf\|_*\|\Bbf\|_{\mathrm{op}}\), yield
\[
\begin{aligned}
\|\Psi(\Xbf)-\Psi(\Ybf)\|_*
&\le
\|\nabla f(\Xbf)-\nabla f(\Ybf)\|_*\|\Xbf\|_{\mathrm{op}}
+
\|\nabla f(\Ybf)\|_*\|\Xbf-\Ybf\|_{\mathrm{op}}
\\
&\le
\brbra{L_*\sqrt{1+\varepsilon}+G_\varepsilon}
\|\Xbf-\Ybf\|_{\mathrm{op}}.
\end{aligned}
\]

For part~\ref{prop:shared-basic-estimates-sampling}, set
\(\Deltabf_i:=
\nabla f(\Xbf;\xi_i)-\nabla f(\Xbf)\).
The oracle assumption gives
\(\mathbb E[\Gbf^{(B)}]=\nabla f(\Xbf)\) and, since \(\Xbf\) is fixed,
\(\mathbb E[\Psibf^{(B)}]=\Psi(\Xbf)\).
The samples are iid and the errors have zero mean, so the cross terms vanish and
\(\mathbb E\|\Gbf^{(B)}-\nabla f(\Xbf)\|_F^2=B^{-2}\sum_{i=1}^B\mathbb E\|\Deltabf_i\|_F^2\le\sigma^2/B\).
Moreover,
\(\Psibf^{(B)}-\Psi(\Xbf)=
\skewop\!\brbra{(\Gbf^{(B)}-\nabla f(\Xbf))\Xbf^\top}\),
and therefore
\(\mathbb E\|\Psibf^{(B)}-\Psi(\Xbf)\|_F^2\le\|\Xbf\|_{\mathrm{op}}^2\mathbb E\|\Gbf^{(B)}-\nabla f(\Xbf)\|_F^2\le(1+\varepsilon)\sigma^2/B\) and \(\mathbb E\|\Psibf^{(B)}-\Psi(\Xbf)\|_*\le\sigma\sqrt{n(1+\varepsilon)/B}\),
where the last inequality follows from
\(\|\Abf\|_*\le\sqrt n\,\|\Abf\|_F\) and Cauchy--Schwarz.

For part~\ref{prop:shared-basic-estimates-coupling}, let
\(\Xbf\in\mathcal S_\varepsilon\).
Since \(C(\Xbf)\in\mathbb S^p\), only the symmetric part of
\(\Xbf^\top\nabla f(\Xbf)\) contributes to the inner product below.
Nuclear--operator duality and \(d(\Xbf)=\|C(\Xbf)\|_{\mathrm{op}}\)
then give
\[
\left|\langle\nabla f(\Xbf),\Xbf C(\Xbf)\rangle\right|
=
\left|
\left\langle
\symop(\Xbf^\top\nabla f(\Xbf)),C(\Xbf)
\right\rangle
\right|
\le
G^{\mathrm{sym}}_\varepsilon d(\Xbf).
\]

\subsection{Polar alignment}
\label{app:momentum-alignment-proof}

\begin{lemma}[Alignment bound under inexact momentum]
\label{lem:momentum-alignment-bound}
For \(\Psibf,\Hbf\in\operatorname{Skew}(n)\), let
\(\Omegabf:=\cpolar(\Hbf)\). Then
\begin{equation}
\label{eq:momentum-alignment-bound}
\langle\Psibf,\Omegabf\rangle
\ge
\|\Psibf\|_*
-
2\|\Hbf-\Psibf\|_*.
\end{equation}
\end{lemma}

Let \(\Omegabf:=\cpolar(\Hbf)\).
Lemma~\ref{lem:matrix-sign-properties} and
nuclear--operator duality imply
\(\langle\Psibf,\Omegabf\rangle=\|\Hbf\|_*+\langle\Psibf-\Hbf,\Omegabf\rangle\ge\|\Hbf\|_*-\|\Psibf-\Hbf\|_*\ge\|\Psibf\|_*-2\|\Hbf-\Psibf\|_*\),
where the last inequality is the reverse triangle inequality.
The same argument covers \(\Hbf=0\).

\subsection{Momentum tracking}
\label{app:generic-momentum-tracking}
\label{subsec:generic-momentum-tracking}

Fix \(\beta\in[0,1)\), let \(\Hbf_0:=\Psibf_0^{(B)}\), and define \(\Hbf_s:=\beta\Hbf_{s-1}+(1-\beta)\Psibf_s^{(B)}\) for \(s\ge1\).

For the stochastic analysis, let
\(\mathcal F_t:=
\sigma(\Xbf_0,\{\xi_{s,i}:0\le s<t,\ 1\le i\le B\})\) and
\(\mathbb E_t[\cdot]:=\mathbb E[\cdot\mid\mathcal F_t]\).
Then \(\Xbf_t\) is \(\mathcal F_t\)-measurable,
as is \(\Hbf_{t-1}\) for \(t\ge1\).
By fresh iid sampling at each iteration, Proposition~\ref{prop:shared-basic-estimates}\ref{prop:shared-basic-estimates-sampling} applies conditionally on \(\mathcal F_t\) at \(\Xbf=\Xbf_t\).

\begin{proposition}
\label{prop:generic-momentum-tracking}
Suppose Assumptions~\ref{ass:operator-nuclear-smoothness}
and~\ref{ass:stochastic-oracle} hold, and
\(\Xbf_0,\ldots,\Xbf_t\in\mathcal S_\varepsilon\) almost surely.
Let \(\nu_{\varepsilon,B}:=
\sigma\sqrt{n(1+\varepsilon)/B}\).
Then
\begin{equation}
\label{eq:generic-momentum-tracking}
\mathbb E\|\Hbf_t-\Psibf_t\|_*
\le
\nu_{\varepsilon,B}
\Bsbra{
\beta^t+
\sqrt{
\frac{1-\beta}{1+\beta}
\brbra{1-\beta^{2t}}
}
}
+
L_{\Psi,\varepsilon}
\sum_{j=0}^{t-1}
\beta^{t-j}
\mathbb E\|\Xbf_{j+1}-\Xbf_j\|_{\mathrm{op}}.
\end{equation}
\end{proposition}

The bound separates a transient contribution from the noisy
initialization, the stationary fluctuation of the EMA estimator,
and a trajectory-variation term caused by tracking the moving
target \(\Psibf_t\).
The proof below combines conditional orthogonality of the mini-batch noise with the Lipschitz continuity of \(\Psibf\).

Define the auxiliary full-gradient momentum by
\(\bar \Hbf_0:=\Psibf_0\) and
\(\bar \Hbf_t:=\beta\bar \Hbf_{t-1}+(1-\beta)\Psibf_t\) for
\(t\ge1\).
The decomposition
\(\Hbf_t-\Psibf_t=(\Hbf_t-\bar \Hbf_t)+(\bar \Hbf_t-\Psibf_t)\)
separates the sampling error from the moving-target lag.

Set \(\zeta_t:=\Psibf_t^{(B)}-\Psibf_t\).
By Proposition~\ref{prop:shared-basic-estimates},
\[
\mathbb E_t[\zeta_t]=0,
\qquad
\mathbb E_t\|\zeta_t\|_F^2
\le
\frac{(1+\varepsilon)\sigma^2}{B},
\qquad
\mathbb E\|\zeta_t\|_*
\le
\nu_{\varepsilon,B}.
\]
Expanding the stochastic and auxiliary momentum recursions gives \(\Hbf_t-\bar\Hbf_t=\beta^t\zeta_0+(1-\beta)\sum_{j=1}^t\beta^{t-j}\zeta_j\).
For \(i<j\), \(\zeta_i\) is \(\mathcal F_j\)-measurable and
\(\mathbb E_j[\zeta_j]=0\); hence
\(\mathbb E\langle\zeta_i,\zeta_j\rangle=0\).
Conditional orthogonality,
\(\|\Abf\|_*\le\sqrt n\|\Abf\|_F\), and Cauchy--Schwarz yield
\[
\begin{aligned}
\mathbb E\|\Hbf_t-\bar \Hbf_t\|_*
&\le
\beta^t\mathbb E\|\zeta_0\|_*
+
(1-\beta)
\sqrt{
n\sum_{j=1}^t
\beta^{2(t-j)}
\mathbb E\|\zeta_j\|_F^2
}
\\
&\le
\nu_{\varepsilon,B}
\Bsbra{
\beta^t+
\sqrt{
\frac{1-\beta}{1+\beta}
\brbra{1-\beta^{2t}}
}
}.
\end{aligned}
\]
For the moving-target component, \(\bar\Hbf_t-\Psibf_t=-\sum_{j=0}^{t-1}\beta^{t-j}(\Psibf_{j+1}-\Psibf_j)\).
The Lipschitz estimate in
Proposition~\ref{prop:shared-basic-estimates} gives
\[
\mathbb E\|\bar \Hbf_t-\Psibf_t\|_*
\le
L_{\Psi,\varepsilon}
\sum_{j=0}^{t-1}
\beta^{t-j}
\mathbb E\|\Xbf_{j+1}-\Xbf_j\|_{\mathrm{op}}.
\]
Combining the two bounds with
\(\Hbf_t-\Psibf_t=(\Hbf_t-\bar \Hbf_t)+(\bar \Hbf_t-\Psibf_t)\)
proves the result.
For \(t=0\), the convolution is empty and the conclusion follows
from \(\Hbf_0-\Psibf_0=\zeta_0\).

\section{Proofs for \algname{}}
\label{app:coupled-proofs}

Following Section~\ref{sec:coupled-convergence}, we prove the safeguards, residual recursion, objective descent, and weighted momentum tracking, then derive the finite-time guarantee and horizon-dependent rates.

\subsection{Proof of Lemma~\ref{lem:coupled-safeguard}}
\label{app:coupled-safeguard-proof}
\label{app:coupled-uniform-safeguard-proof}

\begin{proof}
Write \(\Cbf:=\Xbf^\top\Xbf-\Ibf_p\), \(\Dbf:=\Omegabf\Xbf+\lambda\Xbf\Cbf\), \(d:=\|\Cbf\|_{\mathrm{op}}\), and \(m:=\|\Dbf\|_{\mathrm{op}}\). For \(\Xbf^+:=\Xbf-\tau\Dbf\), the skew-symmetry of \(\Omegabf\) gives
\[
\Cbf^+=\Cbf-2\lambda\tau(\Cbf+\Cbf^2)+\tau^2\Dbf^\top\Dbf.
\]
For \(q:=2\lambda\tau\in[0,1]\), the triangle inequality and \(\|\Cbf^2\|_{\mathrm{op}}\le d^2\) yield
\[
\|\Cbf^+\|_{\mathrm{op}}\le(1-q)d+qd^2+m^2\tau^2=d-2\lambda\tau d(1-d)+m^2\tau^2.
\]
For \(m>0\), requiring the right-hand side to be at most \(\varepsilon\) gives \(m^2\tau^2-2\lambda d(1-d)\tau+d-\varepsilon\le0\). Its nonnegative root, together with \(\tau\le1/(2\lambda)\), is the safeguard threshold in \eqref{eq:coupled-safeguard}. If \(m=0\), then \(\Dbf=0\) and \(\Xbf^+=\Xbf\).

For \(\Xbf_t\in\mathcal S_\varepsilon\), Lemmas~\ref{lem:safe-region-bounds} and~\ref{lem:matrix-sign-properties} give \(m_t\le\sqrt{1+\varepsilon}(1+\lambda\varepsilon)=M_\varepsilon\). For \(m_t>0\), let \(r_\varepsilon(d,m;\lambda)\) denote the nonnegative quadratic root in \eqref{eq:coupled-safeguard}. If \(0\le d\le\varepsilon/2\), then \(r_\varepsilon(d,m;\lambda)\ge\sqrt{\varepsilon/2}/M_\varepsilon\); if \(\varepsilon/2\le d\le\varepsilon\), then \(r_\varepsilon(d,m;\lambda)\ge\lambda\varepsilon(1-\varepsilon)/M_\varepsilon^2\). For \(m_t=0\), the safeguard equals \(1/(2\lambda)\). Hence \eqref{eq:coupled-uniform-safe-threshold} implies \(\eta_t^{\mathrm{safe}}\ge\eta_\varepsilon^\star\). The local safeguard and induction from \(\Xbf_0\in\mathcal S_\varepsilon\) establish pathwise invariance.
\end{proof}

\subsection{Proof of Lemma~\ref{lem:coupled-residual-recursion}}
\label{app:coupled-residual-proof}

By Lemma~\ref{lem:coupled-safeguard}, \(d_t\le\varepsilon\) and \(m_t\le M_\varepsilon\) pathwise. The residual estimate established in its proof yields
\[
d_{t+1}\le d_t-2\lambda\tau_td_t(1-d_t)+m_t^2\tau_t^2\le(1-a_\varepsilon\tau_t)d_t+b_\varepsilon\tau_t^2,
\]
where \(a_\varepsilon=2\lambda(1-\varepsilon)\) and \(b_\varepsilon=M_\varepsilon^2\), proving \eqref{eq:coupled-residual-recursion}. Under \eqref{eq:coupled-target-step-condition}, \(\tau_t\le1/(2\lambda)\) ensures \(1-a_\varepsilon\tau_t\ge\varepsilon>0\). Rearranging the recursion and summing over \(t=0,\ldots,T-1\) gives
\[
a_\varepsilon\sum_{t=0}^{T-1}\tau_td_t\le d_0-d_T+b_\varepsilon\sum_{t=0}^{T-1}\tau_t^2\le d_0+b_\varepsilon\sum_{t=0}^{T-1}\tau_t^2,
\]
which proves \eqref{eq:coupled-feasibility-telescoping}.

\subsection{\texorpdfstring{Objective Descent for \algname{}}{Objective Descent for Landom}}
\label{app:coupled-objective-descent-proof}

\label{subsec:coupled-descent-tracking}

Combining Lemma~\ref{lem:momentum-alignment-bound} with the estimates
in Section~\ref{sec:preliminaries} and
Lemma~\ref{lem:coupled-safeguard} yields the following pathwise
approximate descent property.
\begin{proposition}
\label{prop:coupled-objective-descent}
Suppose that Assumption~\ref{ass:operator-nuclear-smoothness} holds,
\(\Xbf_0\in\mathcal S_\varepsilon\), and the deterministic target
step sizes satisfy \eqref{eq:coupled-target-step-condition}.
Then, for every \(t=0,\ldots,T-1\),
\begin{equation}
\label{eq:coupled-objective-descent}
f(\Xbf_{t+1})
\le
f(\Xbf_t)
-
\tau_t\|\Psibf_t\|_*
+
2\tau_t\|\Hbf_t-\Psibf_t\|_*
+
\lambda G^{\mathrm{sym}}_\varepsilon\tau_td_t
+
\frac{L_*M_\varepsilon^2}{2}\tau_t^2.
\end{equation}
\end{proposition}% <261013 B3-proof: retained original>

By Lemma~\ref{lem:coupled-safeguard}, \(\Xbf_t,\Xbf_{t+1}\in\mathcal S_\varepsilon\). The operator--nuclear smoothness estimate in Appendix~\ref{app:shared-basic-estimates-proof} gives
\[
f(\Xbf_{t+1})\le f(\Xbf_t)-\tau_t\langle\nabla f(\Xbf_t),\Dbf_t\rangle+(L_*/2)\tau_t^2\|\Dbf_t\|_{\mathrm{op}}^2.
\]
By the relative-gradient identity, Lemma~\ref{lem:momentum-alignment-bound}, and the normal-coupling estimate in Appendix~\ref{app:shared-basic-estimates-proof},
\[
\langle\nabla f(\Xbf_t),\Dbf_t\rangle\ge\|\Psibf_t\|_*-2\|\Hbf_t-\Psibf_t\|_*-\lambda G_\varepsilon^{\mathrm{sym}}d_t.
\]
Substitution with \(\|\Dbf_t\|_{\mathrm{op}}\le M_\varepsilon\) proves \eqref{eq:coupled-objective-descent}.

\subsection{Weighted Tracking Accumulation}
\label{app:coupled-tracking-accumulation}

By Algorithm~\ref{alg:landom} and Lemma~\ref{lem:coupled-safeguard}, \(\|\Xbf_{j+1}-\Xbf_j\|_{\mathrm{op}}=\tau_j\|\Dbf_j\|_{\mathrm{op}}\le M_\varepsilon\tau_j\). Applying Proposition~\ref{prop:generic-momentum-tracking} gives
\begin{equation}
\label{eq:coupled-tracking-bound}
\mathbb E\|\Hbf_t-\Psibf_t\|_*\le\nu_{\varepsilon,B}\left[\beta^t+\sqrt{\frac{1-\beta}{1+\beta}(1-\beta^{2t})}\right]+L_{\Psi,\varepsilon}M_\varepsilon\sum_{j=0}^{t-1}\beta^{t-j}\tau_j.
\end{equation}
For constant \(\tau_j\equiv\tau\), the convolution sum is bounded by \(\beta\tau/(1-\beta)\). To obtain the weighted estimate, note that \(2\tau_t\tau_j\le\tau_t^2+\tau_j^2\). Summing the resulting geometric series gives
\[
\sum_{0\le j<t<T}\beta^{t-j}\tau_t\tau_j\le\frac{\beta}{1-\beta}Q_T.
\]
Multiplying \eqref{eq:coupled-tracking-bound} by \(\tau_t\), summing over \(t=0,\ldots,T-1\), and using \(\sqrt{1-\beta^{2t}}\le1\), we obtain
\[
\sum_{t=0}^{T-1}\tau_t\mathbb E\|\Hbf_t-\Psibf_t\|_*\le\nu_{\varepsilon,B}\left[\sum_{t=0}^{T-1}\tau_t\beta^t+\sqrt{\frac{1-\beta}{1+\beta}}W_T\right]+L_{\Psi,\varepsilon}M_\varepsilon\frac{\beta}{1-\beta}Q_T.
\]
Since \(W_T>0\), division by \(W_T\) gives \(\mathcal R_T^{\mathrm{track}}\) in \eqref{eq:coupled-rate-quantities}.

\subsection{Proof of Theorem~\ref{thm:coupled-finite-time}}
\label{app:coupled-finite-time-proof}

\begin{proof}
The pathwise feasibility estimate \eqref{eq:coupled-feasibility-telescoping} gives \(\sum_{t=0}^{T-1}\tau_td_t/W_T\le(d_0+b_\varepsilon Q_T)/(a_\varepsilon W_T)=\mathcal R_T^{\mathrm{feas}}\). Since the steps are deterministic, the corresponding bound also holds after taking expectations.

Summing \eqref{eq:coupled-objective-descent}, using \(f(\Xbf_T)\ge f_\varepsilon^{\inf}\), and taking expectations yield
\[
\frac{\sum_{t=0}^{T-1}\tau_t\mathbb E\|\Psibf_t\|_*}{W_T}\le\frac{\Delta_f}{W_T}+\frac2{W_T}\sum_{t=0}^{T-1}\tau_t\mathbb E\|\Hbf_t-\Psibf_t\|_*+\frac{\lambda G_\varepsilon^{\mathrm{sym}}}{W_T}\sum_{t=0}^{T-1}\tau_t\mathbb E d_t+\frac{L_*M_\varepsilon^2}{2}\frac{Q_T}{W_T}.
\]
Applying the weighted tracking bound in Appendix~\ref{app:coupled-tracking-accumulation} and the preceding feasibility estimate gives \(\mathcal R_T^{\mathrm{stat}}\) in \eqref{eq:coupled-rate-quantities}, proving \eqref{eq:coupled-finite-time-feasibility}.
The random-output identities follow by conditioning on the iterates and using the independence of \(R\).
\end{proof}

\begin{remark}[Safeguard inactivity]
\label{rem:coupled-safeguard-inactivity}
For a horizon-dependent target \(\tau_T=\eta_0T^{-\alpha}\), \(\alpha\in\{1/2,3/4\}\), Lemma~\ref{lem:coupled-safeguard} guarantees \(\tau_t=\tau_T\) pathwise throughout the run whenever \(T\ge\max\{1,(\eta_0/\eta_\varepsilon^\star)^{1/\alpha}\}\); hence the safeguard is inactive for all sufficiently large horizons.
\end{remark}

\subsection{Proof of Corollary~\ref{cor:coupled-stochastic-momentum}}
\label{app:coupled-stochastic-rate-proof}

Since \(d_0=0\) and \(\tau_T=\eta_0T^{-3/4}\), Remark~\ref{rem:coupled-safeguard-inactivity} gives \(\tau_t=\tau_T\) for all sufficiently large \(T\). Hence \(W_T=T\tau_T\), \(Q_T=T\tau_T^2\), and \(\mathcal R_T^{\mathrm{feas}}=\Gamma_\varepsilon\tau_T\). Theorem~\ref{thm:coupled-finite-time} therefore yields \(T^{-1}\sum_{t=0}^{T-1}d_t\le\Gamma_\varepsilon\eta_0T^{-3/4}\) pathwise.

For stationarity, substituting \(\tau_T=\eta_0T^{-3/4}\) and \(\beta_T=1-\beta_0T^{-1/2}\) into \eqref{eq:coupled-rate-quantities}, with \(\sum_{t=0}^{T-1}\beta_T^t\le(1-\beta_T)^{-1}\), gives
\[
\begin{aligned}
\frac1T\sum_{t=0}^{T-1}\mathbb E\|\Psibf_t\|_*\le{}&\left(\frac{\Delta_f}{\eta_0}+2\nu_{\varepsilon,B}\sqrt{\beta_0}+\frac{2L_{\Psi,\varepsilon}M_\varepsilon\eta_0}{\beta_0}\right)T^{-1/4}\\
&+\frac{2\nu_{\varepsilon,B}}{\beta_0}T^{-1/2}+\eta_0\left(\lambda G_\varepsilon^{\mathrm{sym}}\Gamma_\varepsilon+\frac{L_*M_\varepsilon^2}{2}\right)T^{-3/4}.
\end{aligned}
\]
Since \(B\) is fixed independently of \(T\), the coefficients do not depend on \(T\). This proves the stated \(\mathcal O(T^{-1/4})\) stationarity rate.

\subsection{Proof of Corollary~\ref{cor:coupled-deterministic-no-momentum}}
\label{app:coupled-deterministic-rate-proof}

\begin{proof}
For \(\tau_T=\eta_0T^{-1/2}\), Remark~\ref{rem:coupled-safeguard-inactivity} gives \(\tau_t=\tau_T\) throughout the run for all sufficiently large \(T\). Since exact gradients are used, \(\Hbf_0=\Psibf_0\). Setting \(\Zbf_t:=\Hbf_t-\Psibf_t\), the momentum recursion gives \(\Zbf_0=0\) and \(\Zbf_t=-\sum_{j=0}^{t-1}\beta^{t-j}(\Psibf_{j+1}-\Psibf_j)\) for \(t\ge1\). By the relative-gradient Lipschitz estimate proved in Appendix~\ref{app:shared-basic-estimates-proof} and \(\|\Xbf_{j+1}-\Xbf_j\|_{\mathrm{op}}\le M_\varepsilon\tau_T\),
\begin{equation}
\label{eq:coupled-deterministic-tracking}
\|\Zbf_t\|_*\le L_{\Psi,\varepsilon}M_\varepsilon\tau_T\sum_{k=1}^{t}\beta^k\le L_{\Psi,\varepsilon}M_\varepsilon\frac{\beta}{1-\beta}\tau_T.
\end{equation}
This bound also holds at \(t=0\). Thus, exact gradients eliminate sampling noise but not the momentum lag for \(\beta>0\). Lemma~\ref{lem:coupled-residual-recursion} gives \(T^{-1}\sum_{t=0}^{T-1}d_t\le d_0/(a_\varepsilon T\tau_T)+\Gamma_\varepsilon\tau_T\). Summing \eqref{eq:coupled-objective-descent}, using \(f(\Xbf_T)\ge f_\varepsilon^{\inf}\), and applying the tracking and feasibility estimates yield
\[
\frac1T\sum_{t=0}^{T-1}\|\Psibf_t\|_*\le\frac{\Delta_f+\lambda G_\varepsilon^{\mathrm{sym}}d_0/a_\varepsilon}{T\tau_T}+\left(2L_{\Psi,\varepsilon}M_\varepsilon\frac{\beta}{1-\beta}+\lambda G_\varepsilon^{\mathrm{sym}}\Gamma_\varepsilon+\frac{L_*M_\varepsilon^2}{2}\right)\tau_T.
\]
Substituting \(\tau_T=\eta_0T^{-1/2}\) establishes the existing explicit stationarity and feasibility bounds in Corollary~\ref{cor:coupled-deterministic-no-momentum}. Fixed \(\beta<1\) yields the stated \(\mathcal O(T^{-1/2})\) rates; for \(\beta=0\), the momentum lag vanishes.
\end{proof}

\section{Proofs for \algplus{}}
\label{app:separate-steps-proofs}

We prove the safeguards, residual contraction, descent, tracking, and convergence rates for \algplus{}; shared estimates and tail guarantees appear in Appendices~\ref{app:shared-proofs} and~\ref{app:tail-output}.

\subsection{Proof of Lemma~\ref{lem:separate-steps-safeguard}}
\label{app:separate-steps-safeguard-proof}

\begin{proof}
Let \(\Xbf\in\mathcal S_\varepsilon\), \(\Omegabf\in\operatorname{Skew}(n)\) with \(\|\Omegabf\|_{\mathrm{op}}\le1\), and write \(\Cbf:=\Xbf^\top\Xbf-\Ibf_p\), \(d:=\|\Cbf\|_{\mathrm{op}}\), and \(\Ebf:=\gamma\Omegabf\Xbf+\rho\Xbf\Cbf\). Expanding the right-Gram matrix under the update \eqref{eq:separate-steps-displacement} and using \(\|\Ebf\|_{\mathrm{op}}\le\sqrt{1+d}(\gamma+\rho d)\),
\[
d^+\le d-2\rho d(1-d)+(1+d)(\gamma+\rho d)^2.
\]
Requiring the right-hand side to be at most \(\varepsilon\) and solving for \(\gamma\ge0\) gives the safeguard map \eqref{eq:separate-steps-safe-map}.

For the uniform lower bound, define \(\chi(d):=[\varepsilon-d+2\rho d(1-d)]/(1+d)\). Differentiation gives \(\chi'(d)=-(2\rho d^2+4\rho d+1+\varepsilon-2\rho)/(1+d)^2<0\) since \(2\rho<1\). Hence \(h(d):=\sqrt{\chi(d)}-\rho d\) is decreasing on \([0,\varepsilon]\), and
\[
\gamma_{\varepsilon,\rho}^{\mathrm{safe}}(d)\ge h(\varepsilon)=\sqrt{2\rho\varepsilon(1-\varepsilon)/(1+\varepsilon)}-\rho\varepsilon=\gamma_\varepsilon^{\mathrm{unif}}(\rho)>0.
\]
The strict positivity follows from \(\rho\varepsilon(1+\varepsilon)<2(1-\varepsilon)\). For Algorithm~\ref{alg:landom-plus}, \(\gamma_t\le\gamma_{\varepsilon,\rho}^{\mathrm{safe}}(d_t)\) and induction from \(\Xbf_0\in\mathcal S_\varepsilon\) establishes pathwise invariance. If the targets satisfy \(0<\gamma_t^{\mathrm{tar}}\le\gamma_{\varepsilon,\rho}^{\mathrm{adm}}\), then \(\gamma_t^{\mathrm{tar}}\le\gamma_\varepsilon^{\mathrm{unif}}(\rho)\le\gamma_{\varepsilon,\rho}^{\mathrm{safe}}(d_t)\), so \(\gamma_t=\gamma_t^{\mathrm{tar}}\) pathwise. The local argument uses only skew-symmetry and the operator-norm bound, which hold for the compact polar direction by Lemma~\ref{lem:matrix-sign-properties}.
\end{proof}

\subsection{Proof of Lemma~\ref{lem:separate-steps-residual-recursion}}
\label{app:separate-steps-residual-proof}

By the scalar estimate in Appendix~\ref{app:separate-steps-safeguard-proof}, \(d_t\le\varepsilon\), and \(\kappa_\varepsilon=1+\varepsilon\), we obtain
\[
d_{t+1}\le[1-\rho(\Delta_{\varepsilon,\rho}-2\kappa_\varepsilon\gamma_t)]d_t+\kappa_\varepsilon\gamma_t^2.
\]
Since \(\gamma_t\le\gamma_{\varepsilon,\rho}^{\mathrm{ctr}}=\Delta_{\varepsilon,\rho}/(4\kappa_\varepsilon)\), the coefficient of \(d_t\) is at most \(q_{\varepsilon,\rho}=1-\rho\Delta_{\varepsilon,\rho}/2\), proving \eqref{eq:separate-steps-residual-recursion}. Rearranging and summing gives
\begin{equation}
\label{eq:separate-steps-feasibility-telescoping}
a_{\varepsilon,\rho}\sum_{t=0}^{T-1}d_t\le d_0+\kappa_\varepsilon\sum_{t=0}^{T-1}\gamma_t^2.
\end{equation}
Iterating the recursion gives \(d_t\le q_{\varepsilon,\rho}^{\,t}d_0+\kappa_\varepsilon\sum_{j=0}^{t-1}q_{\varepsilon,\rho}^{\,t-1-j}\gamma_j^2\). If \(\gamma_j\le\bar\gamma\), the geometric-series bound and \(1-q_{\varepsilon,\rho}=a_{\varepsilon,\rho}\) yield the residual-floor estimate \eqref{eq:separate-steps-residual-floor}.

\subsection{\texorpdfstring{Objective Descent for \algplus{}}{Objective Descent for Landom+}}
\label{app:separate-steps-objective-descent-proof}

\label{subsec:separate-steps-descent-tracking}

The fixed normal step adds a displacement relative to \algname{}. From \eqref{eq:separate-steps-displacement}, the bound is
\begin{equation}
\|\Xbf_{t+1}-\Xbf_t\|_{\mathrm{op}}
\le
\sqrt{\kappa_\varepsilon}
\brbra{\gamma_t+\rho d_t}.
\label{eq:separate-steps-displacement-bound}
\end{equation}

\begin{proposition}
\label{prop:separate-steps-objective-descent}
Suppose Assumption~\ref{ass:operator-nuclear-smoothness} holds, \(\Xbf_0\in\mathcal S_\varepsilon\), \(0<\rho<\bar\rho_\varepsilon\), and the deterministic target step sizes satisfy \(0<\gamma_t^{\mathrm{tar}}\le\gamma_{\varepsilon,\rho}^{\mathrm{adm}}\). Then
\begin{equation}
f(\Xbf_{t+1})
\le
f(\Xbf_t)
-\gamma_t\|\Psibf_t\|_*
+2\gamma_t\|\Hbf_t-\Psibf_t\|_*
+A_{\varepsilon,\rho}d_t
+L_*\kappa_\varepsilon\gamma_t^2,
\label{eq:separate-steps-objective-descent}
\end{equation}
where $A_{\varepsilon,\rho}
:= \rho G^{\mathrm{sym}}_\varepsilon +
L_*\kappa_\varepsilon\rho^2\varepsilon$.
\end{proposition}% <261013 C3-proof: retained original>

Unlike the \algname{} descent inequality, the residual term is not weighted by the tangential step size; it is controlled by Lemma~\ref{lem:separate-steps-residual-recursion}. Applying the operator--nuclear smoothness estimate established in Appendix~\ref{app:shared-basic-estimates-proof}, the relative-gradient identity, the polar-alignment bound, and the normal-coupling estimate gives
\[
f(\Xbf_{t+1})\le f(\Xbf_t)-\gamma_t\|\Psibf_t\|_*+2\gamma_t\|\Hbf_t-\Psibf_t\|_*+\rho G_\varepsilon^{\mathrm{sym}}d_t+\frac{L_*\kappa_\varepsilon}{2}(\gamma_t+\rho d_t)^2.
\]
Since \((\gamma_t+\rho d_t)^2\le2\gamma_t^2+2\rho^2d_t^2\) and \(d_t^2\le\varepsilon d_t\), the quadratic term is bounded by \(L_*\kappa_\varepsilon\gamma_t^2+L_*\kappa_\varepsilon\rho^2\varepsilon d_t\), proving \eqref{eq:separate-steps-objective-descent}.

\subsection{\algplus{} Tracking and Weighted Accumulation}
\label{app:separate-steps-tracking-accumulation}

Applying Proposition~\ref{prop:generic-momentum-tracking} with the displacement bound \eqref{eq:separate-steps-displacement-bound} yields
\begin{equation}
\label{eq:separate-steps-tracking-bound}
\mathbb E\|\Hbf_t-\Psibf_t\|_*\le\nu_{\varepsilon,B}\left[\beta^t+\sqrt{\frac{1-\beta}{1+\beta}(1-\beta^{2t})}\right]+L_{\Psi,\varepsilon}\sqrt{\kappa_\varepsilon}\sum_{j=0}^{t-1}\beta^{t-j}(\gamma_j+\rho\mathbb E d_j).
\end{equation}
The extra \(\rho\mathbb E d_j\) term comes from momentum lag under the fixed normal step. Multiplying \eqref{eq:separate-steps-tracking-bound} by \(\gamma_t\) and summing over \(t\), the tangential convolution is bounded as in Appendix~\ref{app:coupled-tracking-accumulation}, with \(\gamma_t\) in place of \(\tau_t\). For the normal contribution, \(\gamma_t\le\bar\gamma_T\) and \eqref{eq:separate-steps-feasibility-telescoping} imply
\[
\sum_{0\le j<t<T}\beta^{t-j}\gamma_t\mathbb E d_j\le\frac{\beta\bar\gamma_T}{(1-\beta)a_{\varepsilon,\rho}}(d_0+\kappa_\varepsilon Q_T).
\]
Combining the geometric-series bounds for the initialization and noise terms and dividing by \(G_T>0\) yields \(\mathcal Q_T^{\mathrm{track}}\) in \eqref{eq:separate-steps-rate-quantities}.

\subsection{Proof of Theorem~\ref{thm:separate-steps-finite-time}}
\label{app:separate-steps-finite-time-proof}

\begin{proof}
The pathwise feasibility assertion in \eqref{eq:separate-steps-finite-time-feasibility} follows from \eqref{eq:separate-steps-feasibility-telescoping}. Summing the objective-descent inequality \eqref{eq:separate-steps-objective-descent}, taking expectations, and using \(f(\Xbf_T)\ge f_\varepsilon^{\inf}\) gives
\[
\sum_{t=0}^{T-1}\gamma_t\mathbb E\|\Psibf_t\|_*\le\Delta_f+2\sum_{t=0}^{T-1}\gamma_t\mathbb E\|\Hbf_t-\Psibf_t\|_*+A_{\varepsilon,\rho}\sum_{t=0}^{T-1}\mathbb E d_t+L_*\kappa_\varepsilon Q_T.
\]
Applying \(\mathcal Q_T^{\mathrm{track}}\) in \eqref{eq:separate-steps-rate-quantities} and the residual-sum bound \eqref{eq:separate-steps-feasibility-telescoping}, then dividing by \(G_T>0\), proves the stationarity assertion in \eqref{eq:separate-steps-finite-time-feasibility}.
\end{proof}

\begin{remark}[Horizon-dependent steps and safeguard inactivity]
\label{rem:separate-steps-safeguard-inactivity}
In both schedules above, \(\gamma_T/\rho=(\gamma_0/\rho)T^{-\alpha}\), with \(\alpha\in\{1/2,3/4\}\), while the target \(\gamma_T\) remains constant within each run. The two timescales therefore refer to the relative sizes of the tangential and normal steps as the horizon grows. Lemma~\ref{lem:separate-steps-safeguard} provides a positive admissible cap independent of \(T\), so for all sufficiently large \(T\), the safeguard leaves the target unchanged and \(\gamma_t=\gamma_T\) pathwise throughout the run.
\end{remark}

\subsection{Proof of Corollary~\ref{cor:separate-steps-stochastic-momentum}}
\label{app:separate-steps-stochastic-rate-proof}
Set \(\gamma_T=\gamma_0T^{-3/4}\) and \(\beta_T=1-\beta_0T^{-1/2}\). For all sufficiently large \(T\), Lemma~\ref{lem:separate-steps-safeguard} gives \(\gamma_t=\gamma_T\) pathwise, so \(G_T=T\gamma_T\), \(Q_T=T\gamma_T^2\), and \(\bar\gamma_T=\gamma_T\). Substituting these identities into \eqref{eq:separate-steps-rate-quantities} and \eqref{eq:separate-steps-finite-time-feasibility}, with \(\sum_{t=0}^{T-1}\beta_T^t\le(1-\beta_T)^{-1}\), yields
\begin{equation}
\begin{aligned}
\frac1T\sum_{t=0}^{T-1}\mathbb E\|\Psibf_t\|_*
\le{}&
\frac{\Delta_f}{T\gamma_T}
+
2\nu_{\varepsilon,B}
\Bsbra{
\frac{1}{T(1-\beta_T)}
+
\sqrt{\frac{1-\beta_T}{1+\beta_T}}
}
\\
&+
2L_{\Psi,\varepsilon}\sqrt{\kappa_\varepsilon}
\frac{\beta_T}{1-\beta_T}
\Bsbra{
\gamma_T
+
\frac{\rho d_0}{a_{\varepsilon,\rho}T}
+
\frac{\rho\kappa_\varepsilon}{a_{\varepsilon,\rho}}\gamma_T^2
}
\\
&+
\frac{A_{\varepsilon,\rho}}{a_{\varepsilon,\rho}}
\Bsbra{
\frac{d_0}{T\gamma_T}
+
\kappa_\varepsilon\gamma_T
}
+
L_*\kappa_\varepsilon\gamma_T
=\mathcal O(T^{-1/4}).
\end{aligned}
\label{eq:separate-steps-stochastic-stationarity-explicit}
\end{equation}% <261013 C6-after: retained original>

The prescribed schedules give \((T\gamma_T)^{-1}=\gamma_0^{-1}T^{-1/4}\), \([T(1-\beta_T)]^{-1}=\beta_0^{-1}T^{-1/2}\), and \(\sqrt{(1-\beta_T)/(1+\beta_T)}\le\sqrt{\beta_0}T^{-1/4}\). Moreover, \(\beta_T\gamma_T/(1-\beta_T)=\mathcal O(T^{-1/4})\), while the remaining terms in the displayed bound have the same or smaller order. As \(B\) and \(d_0\) are fixed independently of \(T\), this proves the stationarity rate. For feasibility, \eqref{eq:separate-steps-feasibility-telescoping} gives \(T^{-1}\sum_{t=0}^{T-1}d_t\le d_0/(a_{\varepsilon,\rho}T)+\kappa_\varepsilon\gamma_0^2/(a_{\varepsilon,\rho}T^{3/2})\) pathwise. Since \(R\) is uniform and independent of the algorithmic randomness, \(\mathbb E\|\Psibf_R\|_*=T^{-1}\sum_{t=0}^{T-1}\mathbb E\|\Psibf_t\|_*\) and \(\mathbb E_R[d_R]=T^{-1}\sum_{t=0}^{T-1}d_t\) pathwise; hence \eqref{eq:separate-steps-stochastic-rates} follows. If \(d_0=0\), \eqref{eq:separate-steps-residual-floor} gives \(d_t\le c_{\varepsilon,\rho}^{\mathrm{res}}\gamma_0^2T^{-3/2}\) for every \(0\le t\le T\), proving \eqref{eq:separate-steps-stochastic-pathwise-feasibility}. Finally, \(N(\Xbf_t)\le(p/4)d_t^2\) gives \(\max_{0\le t\le T}N(\Xbf_t)=\mathcal O(pT^{-3})\).

\subsection{Proof of Corollary~\ref{cor:separate-steps-deterministic-no-momentum}}
\label{app:separate-steps-deterministic-rate-proof}

\begin{proof}
By Remark~\ref{rem:separate-steps-safeguard-inactivity}, \(\gamma_t=\gamma_T=\gamma_0T^{-1/2}\) throughout the run for all sufficiently large \(T\). Set \(\Zbf_t:=\Hbf_t-\Psibf_t\). Since exact gradients are used, \(\Zbf_0=0\), and the momentum recursion gives \(\Zbf_t=-\sum_{j=0}^{t-1}\beta^{t-j}(\Psibf_{j+1}-\Psibf_j)\). The relative-gradient Lipschitz estimate established in Appendix~\ref{app:shared-basic-estimates-proof}, together with \eqref{eq:separate-steps-displacement-bound}, yields \(\|\Zbf_t\|_*\le L_{\Psi,\varepsilon}\sqrt{\kappa_\varepsilon}\sum_{j=0}^{t-1}\beta^{t-j}(\gamma_T+\rho d_j)\). Summing over \(t\), interchanging the finite sums, and applying \eqref{eq:separate-steps-feasibility-telescoping}, we obtain
\begin{equation}
\label{eq:separate-steps-deterministic-tracking}
\frac1T\sum_{t=0}^{T-1}\|\Zbf_t\|_*\le L_{\Psi,\varepsilon}\sqrt{\kappa_\varepsilon}\frac{\beta}{1-\beta}\left[\gamma_T+\frac{\rho}{a_{\varepsilon,\rho}}\left(\frac{d_0}{T}+\kappa_\varepsilon\gamma_T^2\right)\right].
\end{equation}
Summing \eqref{eq:separate-steps-objective-descent}, using \(f(\Xbf_T)\ge f_\varepsilon^{\inf}\), and applying \eqref{eq:separate-steps-feasibility-telescoping} and \eqref{eq:separate-steps-deterministic-tracking} give the explicit stationarity bound \eqref{eq:separate-steps-deterministic-stationarity-explicit} after setting \(\gamma_T=\gamma_0T^{-1/2}\) and using \(\mathbb E_R\|\Psibf_R\|_*=T^{-1}\sum_{t=0}^{T-1}\|\Psibf_t\|_*\). The same substitution in \eqref{eq:separate-steps-feasibility-telescoping} gives \eqref{eq:separate-steps-deterministic-feasibility-explicit}, and the uniform-output identities for stationarity and \(\mathbb E_R[d_R]=T^{-1}\sum_{t=0}^{T-1}d_t\) imply \eqref{eq:separate-steps-deterministic-rates}. If \(d_0=0\), \eqref{eq:separate-steps-residual-floor} gives \(d_t\le c_{\varepsilon,\rho}^{\mathrm{res}}\gamma_0^2/T\) for all \(0\le t\le T\), proving \eqref{eq:separate-steps-deterministic-pathwise-feasibility}. Moreover, \(N(\Xbf_t)\le(p/4)d_t^2\) gives \(\max_{0\le t\le T}N(\Xbf_t)\le p(c_{\varepsilon,\rho}^{\mathrm{res}})^2\gamma_0^4/(4T^2)\). When \(\beta=0\), the tracking error vanishes.
\end{proof}

\section{Tail Guarantees from Infeasible Initialization}
\label{app:tail-output}

\subsection{A Shared Affine-Recursion Interface}
\label{app:shared-affine-tail-recursion}

For both algorithms, we first establish an affine-recursion bound and sufficient burn-in indices after which the decaying initialization term is at most the residual floor \(\bar d\).

\begin{lemma}
\label{lem:affine-residual-tail}
Let \(T\ge1\) and let \(\{d_t\}_{t=0}^T\) be a nonnegative sequence satisfying \(d_{t+1}\le qd_t+r\) pathwise for \(t=0,\ldots,T-1\), where \(q\in(0,1)\) and \(r>0\) are deterministic constants. Define \(\bar d:=r/(1-q)\). Then, for every \(t=0,\ldots,T\),
\begin{equation}
\label{eq:affine-residual-bound}
d_t
\le
q^td_0
+
\bar d
\brbra{
1-q^t
}.
\end{equation}

Moreover, define
\(K^{\mathrm{exact}}:=0\) if \(d_0\le\bar d\) and \(K^{\mathrm{exact}}:=\lceil\log(d_0/\bar d)/(-\log q)\rceil\) otherwise.
Then
\begin{equation}
\label{eq:affine-exact-tail-bound}
d_t
\le
2\bar d,
\qquad
K^{\mathrm{exact}}\le t\le T.
\end{equation}

\end{lemma}

\begin{proof}
Iterating the affine recursion and summing the geometric series gives \eqref{eq:affine-residual-bound}. If \(d_0\le\bar d\), then \(d_t\le\bar d\) for every \(t\). Otherwise, the definition of \(K^{\mathrm{exact}}\) implies \(q^{K^{\mathrm{exact}}}d_0\le\bar d\), and \eqref{eq:affine-residual-bound} yields \eqref{eq:affine-exact-tail-bound}.
\end{proof}

\subsection{Tail guarantees for \algname{}}
\label{app:coupled-tail-guarantees-proof}

For a constant target \(\tau_T>0\), define the burn-in index at which the initialization transient falls below the residual floor:
\begin{equation}
\label{eq:coupled-burn-in-index}
K_T:=
\begin{cases}
0,&d_0\le\Gamma_\varepsilon\tau_T,\\
\lceil(a_\varepsilon\tau_T)^{-1}\log(d_0/(\Gamma_\varepsilon\tau_T))\rceil,&d_0>\Gamma_\varepsilon\tau_T,
\end{cases}
\end{equation}% <261013 D2-tail-intro: retained original>

Set \(N_T:=T-K_T\). The following proposition establishes pathwise feasibility and averaged stationarity over the tail of the original run.
\begin{proposition}[Tail guarantees from infeasible initialization]
\label{prop:coupled-tail-guarantees}
Let $\Xbf_0\in\mathcal S_\varepsilon$ be the initial point.
For each schedule below, use $K_T$ and $N_T$ defined above.
All prefactors are independent of $T$. For all sufficiently large $T$,
the following conclusions hold.

\begin{enumerate}[label=\textnormal{(\roman*)},leftmargin=*,widest=ii]
\item
\textnormal{Stochastic momentum.}
Under Assumptions~\ref{ass:operator-nuclear-smoothness}
and~\ref{ass:stochastic-oracle}, with fixed $B$,
$\eta_t^{\mathrm{tar}}\equiv\tau_T=\eta_0T^{-3/4}$ and
$\beta=\beta_T=1-\beta_0T^{-1/2}$, where $\eta_0>0$ and
$\beta_0\in(0,1]$, we have $\tau_t=\tau_T$,
$K_T=\mathcal O(T^{3/4}\log T)$, $T/2\le N_T\le T$, and
\[
\frac1{N_T}\sum_{t=K_T}^{T-1}\mathbb E\|\Psibf_t\|_*
=\mathcal O(T^{-1/4}),
\qquad
\max_{K_T\le t\le T}d_t=\mathcal O(T^{-3/4})
\quad\text{pathwise}.
\]

\item
\textnormal{Deterministic momentum.}
Under Assumption~\ref{ass:operator-nuclear-smoothness}, with exact gradients, fixed \(\beta\in[0,1)\) independent of \(T\), and \(\eta_t^{\mathrm{tar}}\equiv\tau_T=\eta_0T^{-1/2}\), where \(\eta_0>0\), we have \(\tau_t=\tau_T\), \(K_T=\mathcal O(T^{1/2}\log T)\), \(T/2\le N_T\le T\), and
\[
\frac1{N_T}\sum_{t=K_T}^{T-1}\|\Psibf_t\|_*=\mathcal O(T^{-1/2}),\qquad
\max_{K_T\le t\le T}d_t=\mathcal O(T^{-1/2}).
\]

\end{enumerate}
\end{proposition}

\begin{proof}
For either schedule, write \(\tau_T=\eta_0T^{-\alpha}\) with \(\alpha\in\{1/2,3/4\}\). For all sufficiently large \(T\), \(\tau_T\le\eta_\varepsilon^\star\), and Lemma~\ref{lem:coupled-safeguard} ensures \(\tau_t=\tau_T\) pathwise throughout the run.

\emph{Burn-in and tail feasibility.} Applying Lemma~\ref{lem:affine-residual-tail} to \eqref{eq:coupled-residual-recursion} with \(q=1-a_\varepsilon\tau_T\), \(r=b_\varepsilon\tau_T^2\), and \(\bar d=\Gamma_\varepsilon\tau_T\), the inequality \(-\log(1-a_\varepsilon\tau_T)\ge a_\varepsilon\tau_T\) shows that \(K_T\) from \eqref{eq:coupled-burn-in-index} satisfies \(K_T\ge K^{\mathrm{exact}}\). Thus \eqref{eq:affine-exact-tail-bound} gives
\[
\max_{K_T\le t\le T}d_t\le2\Gamma_\varepsilon\tau_T\qquad\text{pathwise}.
\]
For \(K_T>0\), its definition gives \(K_T\le1+(a_\varepsilon\eta_0)^{-1}T^\alpha\log(d_0T^\alpha/(\Gamma_\varepsilon\eta_0))\). Hence \(K_T=\mathcal O(T^\alpha\log T)=o(T)\), also when \(K_T=0\). For all sufficiently large \(T\), \(0\le K_T\le T/2\), so \(T/2\le N_T\le T\). Taking \(\alpha=3/4\) and \(\alpha=1/2\) proves both pathwise tail feasibility claims.

\emph{Tail stationarity.} Since \(K_T,N_T\) are deterministic and \(u_t\ge0\),
\[
\frac1{N_T}\sum_{t=K_T}^{T-1}u_t\le\frac{T}{N_T}\frac1T\sum_{t=0}^{T-1}u_t\le\frac2T\sum_{t=0}^{T-1}u_t.
\]
\emph{(i) Stochastic momentum.} Taking \(u_t=\mathbb E\|\Psibf_t\|_*\) and applying the general full-run estimate in Theorem~\ref{thm:coupled-finite-time} with \(W_T=T\tau_T\) and \(Q_T=T\tau_T^2\) yield
\[
\frac1{N_T}\sum_{t=K_T}^{T-1}\mathbb E\|\Psibf_t\|_*\le\frac{T}{N_T}\mathcal R_T^{\mathrm{stat}}\le2\mathcal R_T^{\mathrm{stat}}.
\]
For \(\tau_T=\eta_0T^{-3/4}\) and \(\beta_T=1-\beta_0T^{-1/2}\), the objective-gap and infeasible-initialization terms in \eqref{eq:coupled-rate-quantities} are bounded by \((\Delta_f+\lambda G_\varepsilon^{\mathrm{sym}}d_0/a_\varepsilon)\eta_0^{-1}T^{-1/4}\). Moreover, \([T(1-\beta_T)]^{-1}=\mathcal O(T^{-1/2})\), \(\sqrt{(1-\beta_T)/(1+\beta_T)}=\mathcal O(T^{-1/4})\), and \(\beta_T\tau_T/(1-\beta_T)=\mathcal O(T^{-1/4})\); the remaining terms are \(\mathcal O(T^{-3/4})\). Fixed \(B,d_0\) therefore give \(\mathcal R_T^{\mathrm{stat}}=\mathcal O(T^{-1/4})\). This invocation uses Theorem~\ref{thm:coupled-finite-time}, not the feasible-initialization Corollary~\ref{cor:coupled-stochastic-momentum}.

\emph{(ii) Deterministic momentum.} Taking \(u_t=\|\Psibf_t\|_*\) and applying the explicit full-run bound \eqref{eq:coupled-deterministic-stationarity-explicit} of Corollary~\ref{cor:coupled-deterministic-no-momentum} gives \(\mathcal O(T^{-1/2})\) on the tail because \(T/N_T\le2\) and \(\beta<1\) is fixed. Both parts concern the original trajectory, with no restart of iterates or momentum at \(K_T\). Together with the pathwise feasibility bound, this completes the proof.
\end{proof}

\subsection{Tail guarantees for \algplus{}}
\label{app:separate-steps-tail-guarantees-proof}

Using \eqref{eq:separate-steps-residual-floor}, choose \(K_T\) so that \(q_{\varepsilon,\rho}^{K_T}d_0\le c_{\varepsilon,\rho}^{\mathrm{res}}\gamma_T^2\):
\begin{equation}
K_T:=
\begin{cases}
0,&d_0\le c_{\varepsilon,\rho}^{\mathrm{res}}\gamma_T^2,\\
\lceil\log(d_0/(c_{\varepsilon,\rho}^{\mathrm{res}}\gamma_T^2))/(-\log q_{\varepsilon,\rho})\rceil,&d_0>c_{\varepsilon,\rho}^{\mathrm{res}}\gamma_T^2,
\end{cases}
\label{eq:separate-steps-feasibility-transient}
\end{equation}% <261013 D3-tail-intro: retained original>

Set \(N_T:=T-K_T\). The following proposition gives feasibility and averaged stationarity over the tail of the same run.

\begin{proposition}[Tail guarantees from infeasible initialization]
\label{prop:separate-steps-tail-guarantees}
Let $\Xbf_0\in\mathcal S_\varepsilon$ be the initial point.
For each schedule below, use $K_T$ and $N_T$ defined above.
All prefactors are independent of $T$.
For all sufficiently large $T$, $K_T=\mathcal O(\log T)$,
$T/2\le N_T\le T$, and the following conclusions hold.

\begin{enumerate}[label=\textnormal{(\roman*)},leftmargin=*,widest=ii]
\item
\textnormal{Stochastic momentum.}
Under the assumptions and schedule of
Corollary~\ref{cor:separate-steps-stochastic-momentum},
\begin{equation}
\frac1{N_T}\sum_{t=K_T}^{T-1}\mathbb E\|\Psibf_t\|_*
=\mathcal O(T^{-1/4}),
\label{eq:separate-steps-tail-stochastic-stationarity}
\end{equation}
and, pathwise,
\begin{equation}
\max_{K_T\le t\le T}d_t=\mathcal O(T^{-3/2}),
\qquad
\max_{K_T\le t\le T}N(\Xbf_t)=\mathcal O(pT^{-3}).
\label{eq:separate-steps-tail-stochastic-feasibility}
\end{equation}

\item
\textnormal{Deterministic momentum.}
Under the assumptions and schedule of
Corollary~\ref{cor:separate-steps-deterministic-no-momentum},
\begin{equation}
\frac1{N_T}\sum_{t=K_T}^{T-1}\|\Psibf_t\|_*=\mathcal O(T^{-1/2}),
\label{eq:separate-steps-tail-deterministic-stationarity}
\end{equation}
while
\begin{equation}
\max_{K_T\le t\le T}d_t=\mathcal O(T^{-1}),
\qquad
\max_{K_T\le t\le T}N(\Xbf_t)=\mathcal O(pT^{-2}).
\label{eq:separate-steps-tail-deterministic-feasibility}
\end{equation}
\end{enumerate}
\end{proposition}

\begin{proof}
For both schedules, write \(\gamma_T=\gamma_0T^{-\alpha}\) with \(\alpha\in\{1/2,3/4\}\). For all sufficiently large \(T\), Lemma~\ref{lem:separate-steps-safeguard} gives \(\gamma_t=\gamma_T\) pathwise throughout the run. Applying Lemma~\ref{lem:affine-residual-tail} to \eqref{eq:separate-steps-residual-recursion} with \(q=q_{\varepsilon,\rho}\) and \(r=\kappa_\varepsilon\gamma_T^2\) gives the residual floor \(\bar d=c_{\varepsilon,\rho}^{\mathrm{res}}\gamma_T^2\). Since \(K_T\) in \eqref{eq:separate-steps-feasibility-transient} is exactly \(K^{\mathrm{exact}}\), \eqref{eq:affine-exact-tail-bound} yields
\begin{equation}
\label{eq:separate-steps-tail-residual}
\max_{K_T\le t\le T}d_t\le2c_{\varepsilon,\rho}^{\mathrm{res}}\gamma_T^2\qquad\text{pathwise}.
\end{equation}
For \(K_T>0\), the definition gives \(K_T\le1+[\log(d_0/(c_{\varepsilon,\rho}^{\mathrm{res}}\gamma_0^2))+2\alpha\log T]/(-\log q_{\varepsilon,\rho})\). As \(q_{\varepsilon,\rho}\in(0,1)\) is independent of \(T\), \(K_T=\mathcal O(\log T)=o(T)\), also when \(K_T=0\). Thus, for all sufficiently large \(T\), \(T/2\le N_T\le T\). Substituting the stochastic and deterministic schedules into \eqref{eq:separate-steps-tail-residual} yields the two tail residual rates; \(N(\Xbf_t)\le(p/4)d_t^2\) gives their squared-penalty counterparts.

For tail stationarity, the nonnegativity of \(u_t\) and \(T/N_T\le2\) give
\[
\frac1{N_T}\sum_{t=K_T}^{T-1}u_t\le\frac{T}{N_T}\frac1T\sum_{t=0}^{T-1}u_t\le\frac2T\sum_{t=0}^{T-1}u_t.
\]
\emph{(i) Stochastic momentum.} Taking \(u_t=\mathbb E\|\Psibf_t\|_*\) and applying the full-run explicit bound \eqref{eq:separate-steps-stochastic-stationarity-explicit} gives \(\mathcal O(T^{-1/4})\) on the tail. In particular, \(d_0/(T\gamma_T)=(d_0/\gamma_0)T^{-1/4}\) does not change the rate.

\emph{(ii) Deterministic momentum.} Since \(R\) is uniform, the randomized-output bound \eqref{eq:separate-steps-deterministic-stationarity-explicit} is identical to the corresponding full-run time-average bound. Taking \(u_t=\|\Psibf_t\|_*\) and applying the full-run explicit bound \eqref{eq:separate-steps-deterministic-stationarity-explicit} gives \(\mathcal O(T^{-1/2})\) on the tail, since \(\beta<1\) is fixed independently of \(T\). Neither the iterates nor the momentum are restarted at \(K_T\). Together with the pathwise feasibility bounds, this completes the proof.
\end{proof}

\section{A Refined Feasibility Bound for Deterministic Landing}
\label{app:additional-structural-results}

\label{app:deterministic-landing-frob-residual}

For deterministic Landing, Proposition~9 of \citet{ablin2024infeasible} controls the squared feasibility measure $N(\Xbf)=\frac14\|\Xbf^\top\Xbf-\Ibf_p\|_F^2$. Tracking the unsquared Frobenius residual directly yields the following refinement.

\begin{proposition}[Frobenius feasibility residual for deterministic Landing]
\label{prop:deterministic-landing-frob-residual}
Let \(\Cbf_k:=\Xbf_k^\top\Xbf_k-\Ibf_p\), \(r_k:=\|\Cbf_k\|_F\), \(N_k:=\frac14r_k^2\), and \(g_k:=\|\grad f(\Xbf_k)\|_F\), and consider \(\Xbf_{k+1}=\Xbf_k-\eta\Lambdabf_k\) with \(\Lambdabf_k:=\grad f(\Xbf_k)+\lambda\Xbf_k\Cbf_k\).
Assume \(0<\eta\le(2\lambda)^{-1}\) and \(r_k\le\varepsilon<1\) for all \(k\).
Let \(\mathcal L\) be the merit function defined in \citet[Equation~(16)]{ablin2024infeasible}, let \(\mathcal L^\star\) be a lower bound for \(\mathcal L\) on the Frobenius safe region \(\{\Xbf:\|\Xbf^\top\Xbf-\Ibf_p\|_F\le\varepsilon\}\), let \(\nu>0\) be the feasibility coefficient in their Proposition~6, and set \(\Delta_{\mathcal L}:=\mathcal L(\Xbf_0)-\mathcal L^\star\).
Further suppose that their two deterministic merit-function estimates in Proposition~9 hold, namely
\(\sum_{k=1}^K g_k^2\le\frac{4\Delta_{\mathcal L}}{\eta},\qquad\sum_{k=1}^K N_k\le\frac{2\Delta_{\mathcal L}}{\eta\nu}\).
Define \(a:=2\lambda(1-\varepsilon)\), \(b:=4\lambda^2(1+\varepsilon)\), and \(C_{\mathrm{res}}:=r_1/(a\eta)+4\Delta_{\mathcal L}/a+2b\Delta_{\mathcal L}/(a\nu)\).
Then, for every \(K\ge1\),
\[
\sum_{k=1}^K r_k\le C_{\mathrm{res}},
\qquad
\min_{\lceil K/2\rceil\le k\le K}r_k\le\frac{2C_{\mathrm{res}}}{K}=\mathcal O(K^{-1}).
\]
In particular, \(\frac1K\sum_{k=1}^Kr_k\le C_{\mathrm{res}}/K=\mathcal O(K^{-1})\).
\end{proposition}

\begin{proof}
The right-Gram identity \eqref{eq:coupled-gram-identity} applies to the deterministic Landing update with \(\tau_t=\eta\), \(\Dbf_t=\Lambdabf_k\), and \(\grad f(\Xbf_k)=\Psibf_k\Xbf_k\) for skew-symmetric \(\Psibf_k\).
In particular, \(\Xbf_k^\top(\lambda\Xbf_k\Cbf_k)+(\lambda\Xbf_k\Cbf_k)^\top\Xbf_k=2\lambda(\Cbf_k+\Cbf_k^2)\).
Since $r_k\le\varepsilon<1$ and
$\eta\le(2\lambda)^{-1}$, a scalar spectral estimate yields
\(\|\Cbf_k-2\eta\lambda(\Cbf_k+\Cbf_k^2)\|_F\le(1-a\eta)r_k\).
Moreover, the tangential and normal components of the Landing field are Frobenius-orthogonal and $\|\Xbf_k\|_{\mathrm{op}}^2\le1+\varepsilon$, so \(\|\Lambdabf_k\|_F^2\le g_k^2+\lambda^2(1+\varepsilon)r_k^2=g_k^2+bN_k\).
Hence
\[
a\eta r_k\le r_k-r_{k+1}+\eta^2g_k^2+b\eta^2N_k.
\]
Summing from $k=1$ to $K$ and applying the two estimates from Proposition~9 of \citet{ablin2024infeasible} gives \(a\eta\sum_{k=1}^K r_k\le r_1+4\eta\Delta_{\mathcal L}+\frac{2b\eta}{\nu}\Delta_{\mathcal L}\), which yields $\sum_{k=1}^K r_k\le C_{\mathrm{res}}$.
\samepage
The average bound is immediate, while \(\min_{\lceil K/2\rceil\le k\le K}r_k\le\frac2K\sum_{k=\lceil K/2\rceil}^Kr_k\le2C_{\mathrm{res}}/K\).
\end{proof}

The proposition bounds average and second-half-best unsquared residuals, not the last-iterate rate $\|\Xbf_K^\top\Xbf_K-\Ibf_p\|_F=\mathcal O(K^{-1})$. Its proof controls quadratic forcing via summability of $g_k^2+bN_k$, which does not automatically extend to a polar-normalized tangential direction.

\raggedbottom
\section{Additional Experimental Results}
\label{app:additional-experimental-results}

\subsection{Residual Scaling}
\label{app:exp1-residual-scaling}

\begin{table}[H]
\centering
\small
\setlength{\tabcolsep}{7pt}
\begin{tabular}{@{}lccc@{}}
\hline
Method & Fitted slope & Predicted scaling & Safe-trigger rate \\
\hline
\algname & \(-0.500004\) & \(T^{-1/2}\) & \(0\) \\
\algplus & \(-1.000004\) & \(T^{-1}\) & \(0\) \\
\hline
\end{tabular}
\caption{\textbf{Residual scaling.} Fits, predicted rates, and safeguard activity (Figure~\ref{fig:exp1-residual-scaling}).}
\label{tab:exp1-residual-scaling}
\end{table}

\subsection{\texorpdfstring{Additional Online PCA and Mini-Batch Real-Data ICA Results}{Additional Online PCA and Mini-Batch Real-Data ICA Results}}
\label{app:exp2-online-pca}

\begin{figure}[!htb]
	\centering
	\includegraphics[width=0.94\linewidth]{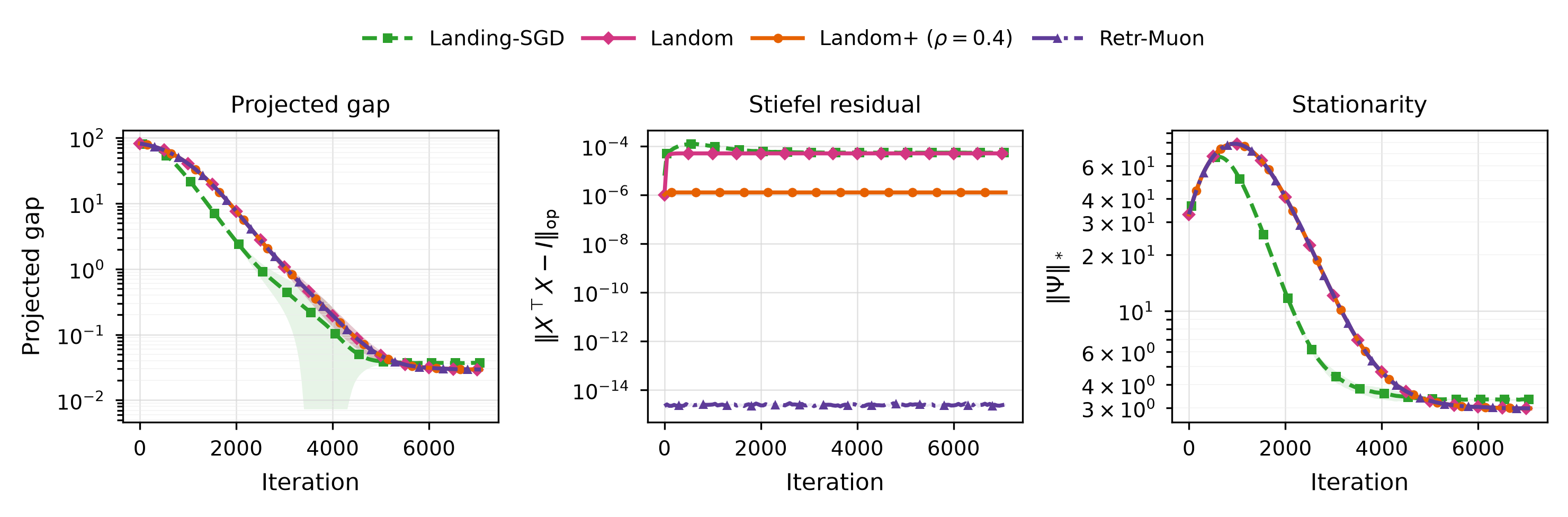}
	\caption{\textbf{Easy spectrum.} Additional online PCA results under
		the easy spectrum.}
	\label{fig:exp2-easy-appendix}
\end{figure}

\begin{figure}[!htb]
\centering
\includegraphics[width=0.94\linewidth]{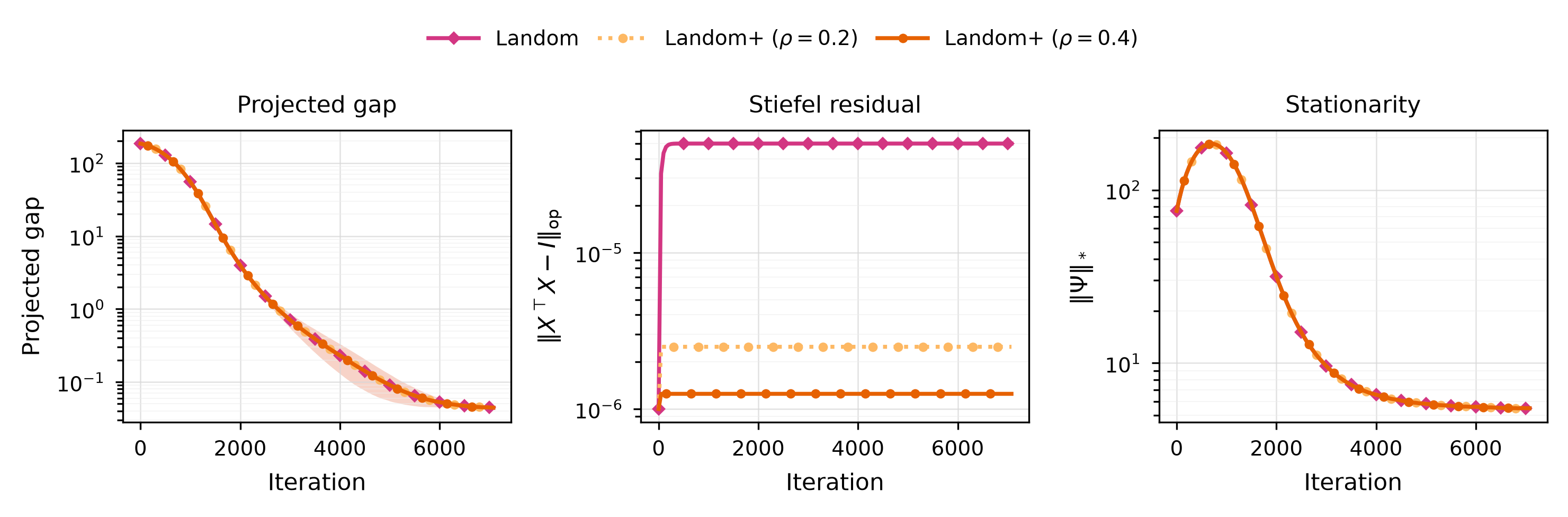}
\caption{\textbf{Normal-step sensitivity.} Effect of \(\rho\) on \algplus{} in stochastic online PCA.}
\label{fig:exp2-rho-robustness}
\end{figure}

\label{app:exp3-ica-additional}

\FloatBarrier
\noindent\textbf{Mini-batch real-data ICA.} On the same real-data ICA task (mini-batch size \(B=256\)), we compare Landing-SGD, stochastic \algname{}, and stochastic \algplus{}. All curves in Figure~\ref{fig:exp3-ica-minibatch-selected} use \(f_{\mathrm{ref}}:=\min_{\Xbf\in\mathcal R_{\mathrm{mb}}}f(\Pi_{\mathrm{St}}(\Xbf))\), with \(\mathcal R_{\mathrm{mb}}\) the formal mini-batch reference pool; this value is not a certified global optimum.
Figure~\ref{fig:exp3-ica-minibatch-selected} shows the same qualitative
behavior as in the full-batch experiment: the polar-factor methods
achieve smaller objective gaps after projection and stationarity measures than
Landing-SGD, while \algplus\ further improves feasibility
relative to \algname.

\begin{figure}[htbp]
\centering
\includegraphics[width=0.96\linewidth]{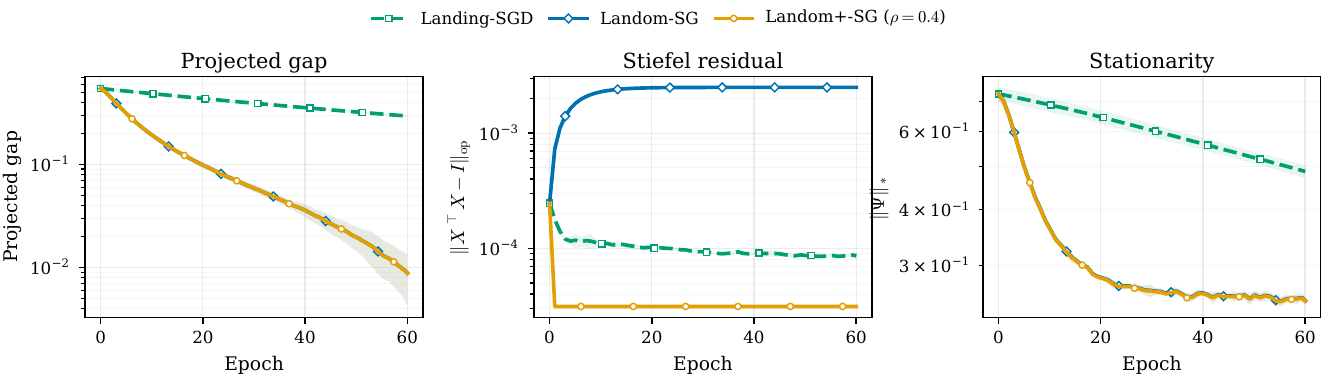}
\caption{\textbf{Mini-batch real-data ICA.} Objective gap after projection,
Stiefel residual, and nuclear-norm stationarity with mini-batch size
\(B=256\).}
\label{fig:exp3-ica-minibatch-selected}
\end{figure}

\FloatBarrier

%% file: arxiv-landom.bbl
\begin{thebibliography}{36}
\providecommand{\natexlab}[1]{#1}
\providecommand{\url}[1]{\texttt{#1}}
\expandafter\ifx\csname urlstyle\endcsname\relax
  \providecommand{\doi}[1]{doi: #1}\else
  \providecommand{\doi}{doi: \begingroup \urlstyle{rm}\Url}\fi

\bibitem[Ablin and Peyr\'e(2022)]{pmlr-v151-ablin22a}
Pierre Ablin and Gabriel Peyr\'e.
\newblock Fast and accurate optimization on the orthogonal manifold without
  retraction.
\newblock In \emph{Proceedings of The 25th International Conference on
  Artificial Intelligence and Statistics}, volume 151 of \emph{Proceedings of
  Machine Learning Research}, pages 5636--5657. PMLR, 2022.

\bibitem[Ablin et~al.(2024)Ablin, Vary, Gao, and Absil]{ablin2024infeasible}
Pierre Ablin, Simon Vary, Bin Gao, and Pierre-Antoine Absil.
\newblock Infeasible deterministic, stochastic, and variance-reduction
  algorithms for optimization under orthogonality constraints.
\newblock \emph{Journal of Machine Learning Research}, 25\penalty0
  (389):\penalty0 1--38, 2024.

\bibitem[Absil and Malick(2012)]{absil2012projection}
P.-A. Absil and J{\'e}r{\^o}me Malick.
\newblock Projection-like retractions on matrix manifolds.
\newblock \emph{SIAM Journal on Optimization}, 22\penalty0 (1):\penalty0
  135--158, 2012.
\newblock \doi{10.1137/100802529}.

\bibitem[Absil et~al.(2008)Absil, Mahony, and Sepulchre]{absil2008optimization}
P.-A. Absil, Robert Mahony, and Rodolphe Sepulchre.
\newblock \emph{Optimization Algorithms on Matrix Manifolds}.
\newblock Princeton University Press, Princeton, NJ, 2008.
\newblock \doi{10.1515/9781400830244}.

\bibitem[Amsel et~al.(2026)Amsel, Persson, Musco, and Gower]{amsel2026polar}
Noah Amsel, David Persson, Christopher Musco, and Robert~M. Gower.
\newblock The {Polar Express}: Optimal matrix sign methods and their
  application to the {Muon} algorithm.
\newblock In \emph{International Conference on Learning Representations}, 2026.

\bibitem[An et~al.(2026)An, Li, Goldfarb, and Ma]{an2026demystifying}
Kang An, Jiaxiang Li, Donald Goldfarb, and Shiqian Ma.
\newblock Demystifying manifold constraints in {LLM} pre-training.
\newblock \emph{arXiv preprint arXiv:2605.04418}, 2026.

\bibitem[Bansal et~al.(2018)Bansal, Chen, and Wang]{bansal2018orthogonality}
Nitin Bansal, Xiaohan Chen, and Zhangyang Wang.
\newblock Can we gain more from orthogonality regularizations in training deep
  networks?
\newblock In \emph{Advances in Neural Information Processing Systems},
  volume~31, pages 4261--4271, 2018.

\bibitem[Bernstein(2025)]{bernstein2025manifolds}
Jeremy Bernstein.
\newblock Modular manifolds.
\newblock \emph{Thinking Machines Lab: Connectionism}, 2025.
\newblock \doi{10.64434/tml.20250926}.

\bibitem[Bernstein and Newhouse(2024)]{bernstein2024old}
Jeremy Bernstein and Laker Newhouse.
\newblock Old optimizer, new norm: An anthology.
\newblock In \emph{OPT 2024: Optimization for Machine Learning}, 2024.

\bibitem[Buchanan(2025)]{buchanan2025mmuonadmm}
Sam Buchanan.
\newblock A faster {Manifold Muon} with {ADMM}, 2025.

\bibitem[Cutkosky and Mehta(2020)]{pmlr-v119-cutkosky20b}
Ashok Cutkosky and Harsh Mehta.
\newblock Momentum improves normalized {SGD}.
\newblock In \emph{Proceedings of the 37th International Conference on Machine
  Learning}, volume 119 of \emph{Proceedings of Machine Learning Research},
  pages 2260--2268. PMLR, 13--18 Jul 2020.

\bibitem[Gao et~al.(2019)Gao, Liu, Chen, and Yuan]{gao2019parallelizable}
Bin Gao, Xin Liu, Xiaojun Chen, and Ya-xiang Yuan.
\newblock Parallelizable algorithms for optimization problems with
  orthogonality constraints.
\newblock \emph{SIAM Journal on Scientific Computing}, 41\penalty0
  (3):\penalty0 A1949--A1983, 2019.

\bibitem[Gao et~al.(2022)Gao, Vary, Ablin, and Absil]{gao2022optimization}
Bin Gao, Simon Vary, Pierre Ablin, and P.-A. Absil.
\newblock Optimization flows landing on the {Stiefel} manifold.
\newblock \emph{IFAC-PapersOnLine}, 55\penalty0 (30):\penalty0 25--30, 2022.
\newblock \doi{10.1016/j.ifacol.2022.11.023}.

\bibitem[Gramfort et~al.(2013)Gramfort, Luessi, Larson, Engemann, Strohmeier,
  Brodbeck, Goj, Jas, Brooks, Parkkonen, and
  H{\"a}m{\"a}l{\"a}inen]{gramfort2013mne}
Alexandre Gramfort, Martin Luessi, Eric Larson, Denis~A. Engemann, Daniel
  Strohmeier, Christian Brodbeck, Roman Goj, Mainak Jas, Teon Brooks, Lauri
  Parkkonen, and Matti~S. H{\"a}m{\"a}l{\"a}inen.
\newblock {MEG} and {EEG} data analysis with {MNE-Python}.
\newblock \emph{Frontiers in Neuroscience}, 7:\penalty0 267, 2013.
\newblock \doi{10.3389/fnins.2013.00267}.

\bibitem[Hazan et~al.(2015)Hazan, Levy, and Shalev-Shwartz]{hazan2015beyond}
Elad Hazan, Kfir~Y. Levy, and Shai Shalev-Shwartz.
\newblock Beyond convexity: Stochastic quasi-convex optimization.
\newblock In \emph{Advances in Neural Information Processing Systems},
  volume~28, pages 1594--1602, 2015.

\bibitem[Hotelling(1933)]{hotelling1933principal}
Harold Hotelling.
\newblock Analysis of a complex of statistical variables into principal
  components.
\newblock \emph{Journal of Educational Psychology}, 24\penalty0 (6):\penalty0
  417--441, 1933.
\newblock \doi{10.1037/h0071325}.

\bibitem[Hotelling(1936)]{hotelling1936relations}
Harold Hotelling.
\newblock Relations between two sets of variates.
\newblock \emph{Biometrika}, 28\penalty0 (3-4):\penalty0 321--377, 1936.
\newblock \doi{10.1093/biomet/28.3-4.321}.

\bibitem[Hyv{\"a}rinen(1999)]{hyvarinen1999fast}
Aapo Hyv{\"a}rinen.
\newblock Fast and robust fixed-point algorithms for independent component
  analysis.
\newblock \emph{IEEE Transactions on Neural Networks}, 10\penalty0
  (3):\penalty0 626--634, 1999.
\newblock \doi{10.1109/72.761722}.

\bibitem[Jordan et~al.(2024)Jordan, Jin, Boza, You, Cesista, Newhouse, and
  Bernstein]{jordan2024muon}
Keller Jordan, Yuchen Jin, Vlado Boza, Jiacheng You, Franz Cesista, Laker
  Newhouse, and Jeremy Bernstein.
\newblock {Muon}: An optimizer for hidden layers in neural networks, 2024.

\bibitem[Krizhevsky(2009)]{krizhevsky2009learning}
Alex Krizhevsky.
\newblock Learning multiple layers of features from tiny images.
\newblock Technical report, University of Toronto, 2009.
\newblock URL
  \url{https://www.cs.toronto.edu/~kriz/learning-features-2009-TR.pdf}.

\bibitem[Levy(2016)]{levy2016power}
Kfir~Y. Levy.
\newblock The power of normalization: Faster evasion of saddle points.
\newblock \emph{arXiv preprint arXiv:1611.04831}, 2016.

\bibitem[Li et~al.(2026)Li, Pandey, Sah, Han, Mostajeran, Jawanpuria, and
  Mishra]{li2026intrinsic}
Yibang Li, Bihari~Lal Pandey, Ravi Sah, Andi Han, Cyrus Mostajeran, Pratik
  Jawanpuria, and Bamdev Mishra.
\newblock Intrinsic {Muon}: Spectral optimization on riemannian matrix
  manifolds.
\newblock \emph{arXiv preprint arXiv:2605.09238}, 2026.

\bibitem[Lu et~al.(2026)Lu, Mei, and Xiao]{lu2026variance}
Zhaosong Lu, Sanyou Mei, and Yifeng Xiao.
\newblock Variance-reduced first-order methods for deterministically
  constrained stochastic nonconvex optimization with strong convergence
  guarantees.
\newblock \emph{SIAM Journal on Optimization}, 36\penalty0 (1):\penalty0 1--31,
  2026.
\newblock \doi{10.1137/24M1693933}.

\bibitem[Park et~al.(2025)Park, Kang, Lee, Lee, Kim, and
  Lee]{park-etal-2025-riemannian}
JuneYoung Park, Minjae Kang, Seongbae Lee, Haegang Lee, Seongwan Kim, and Jaeho
  Lee.
\newblock {R}iemannian optimization for {L}o{RA} on the {Stiefel} manifold.
\newblock In \emph{Findings of the Association for Computational Linguistics:
  EMNLP 2025}, pages 20971--20985, Suzhou, China, November 2025. Association
  for Computational Linguistics.
\newblock \doi{10.18653/v1/2025.findings-emnlp.1143}.

\bibitem[Pethick et~al.(2025)Pethick, Xie, Antonakopoulos, Zhu, Silveti-Falls,
  and Cevher]{pmlr-v267-pethick25a}
Thomas Pethick, Wanyun Xie, Kimon Antonakopoulos, Zhenyu Zhu, Antonio
  Silveti-Falls, and Volkan Cevher.
\newblock Training deep learning models with norm-constrained {LMO}s.
\newblock In \emph{Proceedings of the 42nd International Conference on Machine
  Learning}, volume 267 of \emph{Proceedings of Machine Learning Research},
  pages 49069--49104. PMLR, 13--19 Jul 2025.

\bibitem[Schechtman et~al.(2023)Schechtman, Tiapkin, Muehlebach, and
  Moulines]{schechtman2023orthogonal}
Sholom Schechtman, Daniil Tiapkin, Michael Muehlebach, and {\'E}ric Moulines.
\newblock Orthogonal directions constrained gradient method: From non-linear
  equality constraints to {Stiefel} manifold.
\newblock In \emph{Proceedings of Thirty Sixth Conference on Learning Theory},
  volume 195 of \emph{Proceedings of Machine Learning Research}, pages
  1228--1258. PMLR, 2023.

\bibitem[Sch{\"o}nemann(1966)]{schonemann1966generalized}
Peter~H. Sch{\"o}nemann.
\newblock A generalized solution of the orthogonal procrustes problem.
\newblock \emph{Psychometrika}, 31\penalty0 (1):\penalty0 1--10, 1966.
\newblock \doi{10.1007/BF02289451}.

\bibitem[Shen et~al.(2026)Shen, Huang, Huang, Shen, and Zhang]{shen2026muon}
Wei Shen, Ruichuan Huang, Minhui Huang, Cong Shen, and Jiawei Zhang.
\newblock On the convergence analysis of {Muon}.
\newblock \emph{arXiv preprint arXiv:2505.23737}, 2026.
\newblock Version 3, revised July 2026.

\bibitem[Solonko et~al.(2026)Solonko, Molozhavenko, and
  Rakhuba]{solonko2026skewon}
Mikhail Solonko, Alexander Molozhavenko, and Maxim Rakhuba.
\newblock {Muon} on the {Stiefel} manifold admits an exact closed-form update.
\newblock \emph{arXiv preprint arXiv:2608.06218}, 2026.

\bibitem[Wang et~al.(2020)Wang, Chen, Chakraborty, and Yu]{wang2020orthogonal}
Jiayun Wang, Yubei Chen, Rudrasis Chakraborty, and Stella~X. Yu.
\newblock Orthogonal convolutional neural networks.
\newblock In \emph{2020 IEEE/CVF Conference on Computer Vision and Pattern
  Recognition (CVPR)}, pages 11502--11512, June 2020.
\newblock \doi{10.1109/CVPR42600.2020.01152}.

\bibitem[Wen and Yin(2013)]{wen2013feasible}
Zaiwen Wen and Wotao Yin.
\newblock A feasible method for optimization with orthogonality constraints.
\newblock \emph{Mathematical Programming}, 142\penalty0 (1--2):\penalty0
  397--434, 2013.
\newblock \doi{10.1007/s10107-012-0584-1}.

\bibitem[Xiao et~al.(2022)Xiao, Liu, and Yuan]{xiao2022class}
Nachuan Xiao, Xin Liu, and Ya-xiang Yuan.
\newblock A class of smooth exact penalty function methods for optimization
  problems with orthogonality constraints.
\newblock \emph{Optimization Methods and Software}, 37\penalty0 (4):\penalty0
  1205--1241, 2022.

\bibitem[Xiao et~al.(2024)Xiao, Liu, and Toh]{xiao2024dissolving}
Nachuan Xiao, Xin Liu, and Kim-Chuan Toh.
\newblock Dissolving constraints for {Riemannian} optimization.
\newblock \emph{Mathematics of Operations Research}, 49\penalty0 (1):\penalty0
  366--397, 2024.
\newblock \doi{10.1287/moor.2023.1360}.

\bibitem[Xiong et~al.(2026)Xiong, Gao, and Absil]{xiong2026second}
Xinhui Xiong, Bin Gao, and P.-A. Absil.
\newblock A second-order method landing on the {Stiefel} manifold via
  {Newton--Schulz} iteration.
\newblock \emph{arXiv preprint arXiv:2605.02838}, 2026.
\newblock \doi{10.48550/arXiv.2605.02838}.

\bibitem[Yang and Lai(2026)]{yang2026manifold}
Kaiwei Yang and Lexiao Lai.
\newblock Manifold constrained steepest descent for smooth and closed-set
  optimization.
\newblock \emph{arXiv preprint arXiv:2601.21487}, 2026.

\bibitem[Zhang et~al.(2026)Zhang, Hu, Lai, Lin, and Wen]{zhang2026retraction}
Yuan Zhang, Jiang Hu, Zhijian Lai, Lin Lin, and Zaiwen Wen.
\newblock Retraction-free optimization over the {Stiefel} manifold for the
  {LoRA} fine-tuning.
\newblock \emph{IEEE Transactions on Pattern Analysis and Machine
  Intelligence}, pages 1--12, 2026.
\newblock \doi{10.1109/TPAMI.2026.3711595}.

\end{thebibliography}
